\documentclass[11pt]{article}

\usepackage{amsfonts}
\usepackage{amsmath}
\usepackage{amsthm}
\usepackage{amssymb}
\usepackage{enumitem}
\usepackage{color}
\usepackage{url}
\usepackage{comment}
\usepackage{geometry}
\usepackage{bbm,bm}
\usepackage{algorithm}
\usepackage[noend]{algpseudocode}
\usepackage{csvsimple}
\usepackage{float}

\makeatletter
\def\BState{\State\hskip-\ALG@thistlm}
\makeatother
\algdef{SE}[DOWHILE]{Do}{doWhile}{\algorithmicdo}[1]{\algorithmicwhile\ #1}%

\renewcommand{\labelenumi}{\rm{(\roman{enumi})}}

\usepackage{tikz}
\tikzstyle{vertex}=[circle, draw, inner sep=2pt, fill=white]

\newtheorem{thm}{Theorem}[section]
\newtheorem{cor}[thm]{Corollary}
\newtheorem{lem}[thm]{Lemma}
\newtheorem{prop}[thm]{Proposition}

\theoremstyle{definition}
\newtheorem{defn}[thm]{Definition}

\newtheorem{ass}[thm]{Assumption}

\newtheorem{rem}[thm]{Remark}
\newtheorem{exa}[thm]{Example}
\numberwithin{equation}{section}
\numberwithin{thm}{section}

\newcommand{\R}{\mathbb{R}}

\newcommand{\XX}{\mathbb{X}}
\renewcommand{\SS}{\mathbb{S}}
\newcommand{\YY}{\mathbb{Y}}

\DeclareMathOperator*{\motimes}{\text{\raisebox{0.25ex}{\scalebox{0.7}{$\bigotimes$}}}}

\newcommand{\dbra}[1]{[\kern-0.15em[ #1 ]\kern-0.15em]}
\newcommand{\dbraco}[1]{[\kern-0.15em[ #1 [\kern-0.15em[}
\newcommand{\dbraoc}[1]{]\kern-0.15em] #1 ]\kern-0.15em]}
\newcommand{\dbraoo}[1]{]\kern-0.15em] #1 [\kern-0.15em[}

\newcommand{\Ecal}{\mathcal{E}}
\newcommand{\Fcal}{\mathcal{F}}

\newcommand{\Pcal}{\mathcal{P}}

\newcommand{\Xcal}{\mathcal{X}}
\newcommand{\Ycal}{\mathcal{Y}}
\newcommand{\FF}{\mathbb{F}}

\newcommand{\RR}{\mathbb{R}}
\newcommand{\NN}{\mathbb{N}}

\DeclareMathOperator{\argmin}{argmin}

\makeatletter
\newcommand{\LeftEqNo}{\let\veqno\@@leqno}
\makeatother

\allowdisplaybreaks

\begin{document}

\title{Cournot-Nash equilibrium and optimal transport in a dynamic setting}
\author{
	Beatrice Acciaio\thanks{Department of Mathematics. ETH Zurich, Switzerland.}
	\and
	Julio Backhoff-Veraguas\thanks{ Department of Applied Mathematics. University of Twente, The Netherlands.}
	\and Junchao Jia\thanks{Department of Statistics. London School of Economics, UK.}
}

\date{\today}

\maketitle

\abstract{
We consider a large population dynamic game in discrete time. The peculiarity of the game is that players are characterized by time-evolving types, and so reasonably their actions should not anticipate the future values of their types. When interactions between players are of mean-field kind, we relate Nash equilibria for such games to an asymptotic notion of dynamic Cournot-Nash equilibria. Inspired by the works of Blanchet and Carlier~\cite{BC16} for the static situation, we interpret dynamic Cournot-Nash equilibria in the light of causal optimal transport theory. Further specializing to games of potential type, we establish existence, uniqueness and characterization of equilibria. Moreover we develop, for the first time, a numerical scheme for causal optimal transport, which is then leveraged in order to compute dynamic Cournot-Nash equilibria. This is illustrated in a detailed case study of a congestion game. 
}

\section{Introduction}
We consider a discrete-time dynamic game played by $N$ agents and study its behaviour as $N$ goes to infinity.
Agents take actions in time in order to minimize a cost function that depends on the whole population of players in a mean-field fashion. Agents may have different characteristics or preferences, which define their own ``type" and may change (progressively) in time. The types play a crucial role in the choice of an agent's dynamic actions, since the latter can only rely on the partial knowledge of types to date. Let us illustrate this with an example.

\begin{exa}\label{ex_intro}
Agents represent different delivery services. Each of these  must visit a number of sites in a given order (for instance, supermarkets or costumers). The decision of each agent is whether to take the quick or the slow road in between sites. The mean-field character of the situation arises because of congestion on the roads, which is caused by the population of agents. At each site an agent collects a parcel, to be delivered to the next site. Parcels may be tagged as ``express" or ``normal", and there is a penalty for the slow delivery of an express parcel. The type of the agent is the sequence of parcel tags. In the dynamic game we consider, an agent may solely base her actions (ie.\ taking quick or slow road) on the observed sequence of parcel tags, but is uninformed about the tags of future parcels. By contrast, in a static game, the full sequence of parcel tags is available from the start of the game. In Section~\ref{sect.ex} we illustrate how this leads to a different behaviour of the agents when we compare the dynamic and static games. 
\end{exa}

We study Nash equilibria for such dynamic games by considering an asymptotic formulation of the problem (namely where we think of a continuum of infinitely small players), whose solutions we refer to as \emph{dynamic Cournot-Nash equilibria}. The relevance of this problem for the finite population game is justified by limiting results that we prove under suitable assumptions; see Propositions \ref{Prop eNash} and \ref{Prop conv eq}.

The first main contribution of the present paper is an equivalent reformulation of the asymptotic problem in terms of optimal transport. In doing so we are strongly
inspired by the work of Blanchet and Carlier \cite{BC16} for static games. Crucially, in our dynamic setting actions take place in time and cannot anticipate the evolution of agents' types. This imposes a constraint in the optimal transport setting, which is known in the literature under the name of \emph{causality}, c.f.\ \cite{BBLZ}. 

Our second main contribution is to establish existence, and a characterization, of dynamic Cournot-Nash equilibria for the prominent subclass of \emph{potential games}. This for example allows us to include congestion effects in the cost function. Games of this type are particularly tractable since, as it is well-known, they can be recast as variational problems. Specifically, the search for dynamic Cournot-Nash equilibria boils down to solving optimization problems over causal couplings. 

Finally, we provide an algorithm to compute optimal causal couplings, which is the first numerical method for causal optimal transport and therefore a result of its own interest. In the present context, this allows us to compute dynamic Cournot-Nash equilibria, as well as cooperative equilibria and the Price of Anarchy in the asymptotic formulation. Thanks to our convergence results, this leads to the computation of approximate equilibria for the $N$-player game. Given its generality, we expect the algorithm to be used in other applications of causal optimal transport, e.g.\ in implementing the discretization scheme proposed in \cite{ABC19} or in the machine learning context currently under development in \cite{AMX}.

For illustration, we implement the aforementioned algorithm in a toy model, akin to Example \ref{ex_intro}. This model is used to further highlight the difference between the notions of dynamic and static Cournot-Nash equilibria.

\subsection{Related literature}

{ 

From the mathematical perspective, our formulation is closely related to mean field games (MFG) in a discrete-time setting (see e.g.\ \cite{GMS10}). For this parallel, the different types of agents considered in our setup correspond to different subpopulations of players in the MFG (see Remark \ref{rem:subpopu} below for more details). 
The theory of mean field games aims at studying dynamic games as the number of agents tends to infinity. It was established independently in the mathematical community by Lasry and Lions \cite{lasry2006jeux,LL07}, and in the engineering community by Huang, Malham\'e and Caines \cite{CHM06,huang2007large}, 
and has since seen a burst in activity, as e.g.\ documented in the recent monograph by Carmona and Delarue \cite{CD18}. See also Cardialaguet's \cite{Car10}, based on P.L. Lions' lectures at Coll\'ege de France, for seminal results on mean field games. 
The key assumption is that players are symmetric and rational, and the idea is to approximate large $N$-player systems by studying the behaviour as $N \rightarrow \infty$.

Various efforts have been made to rigorously prove that, when the number of players tends to infinity, the original $N$-player system indeed converges to the MFG limit. In the original paper \cite{LL07}, the authors proved such a result for the stationary case, while convergence is established for dynamical MFG in other specific cases by e.g. \cite{bardi2012explicit} for the linear quadratic case and \cite{ferreira2014convergence,cecchin2020probabilistic} for finite state MFG, among others.
Several papers were devoted to deal with more general cases based on different notions of convergence, see e.g.\ \cite{cardaliaguet2015master,L16,nutz2020convergence,cecchin2019convergence}. 
Given that MFG assumes the players to be competitive, the Nash equilibria may not result in optimal social cost. This leads to the discussion of (in)efficiency of the equilibria and originates the concept of Price of Anarchy (PoA). \cite{graber2016linear} studied the inefficiency of Nash equilibria for the linear quadratic case, and \cite{carmona2019price} defined and analyzed the PoA for MFG systems. 
See also \cite{BZ19} for a thorough analysis of the price of anarchy in a continuous-time terminal ranking game with non-local mean field effect, and a connection to the so-called Schr\"odinger problem.

The articles which are the closest to ours are the works by Blanchet and Carlier \cite{BC14a,BC16} where, building on the seminal contribution of Mas-Colell \cite{M84}, a connection between static Cournot-Nash equilibria and optimal transport is developed. From a probabilistic perspective, large static anonymous games have been recently studied by Lacker and Ramanan in \cite{LR19}, with an emphasis on large deviations and the asymptotic behaviour of the Price of Anarchy. We also refer to this paper for a thorough review on the (vast) game theoretic literature. The crucial difference between the above papers and the present one is the \emph{dynamic} nature of the game we consider here. 
It is also worth mentioning that \cite{benamou2017variational} also considers a variational formulation to study competitive games with mean field effect. Similar to our paper, the cost is separable and of potential type, and this leads to a formulation in terms of an optimal transport cost plus an energy functional. A possible extension is provided in \cite{BC14b} where also the non-potential and non-separable cases are considered: It could be of interest to develop in the future a corresponding picture in the dynamic setting.
 
To deal with our dynamic setting, we use tools from causal optimal transport (COT) rather than classical optimal transport. In a nutshell, COT is a relative of the optimal transport problem where an extra constraint, which takes into account the arrow of time (filtrations), is added. This in turn is crucial to ensure, in our application, the adaptedness of players' actions to their types in a dynamic framework.
The theory of COT, used to reformulate our asymptotic equilibrium problem, has been developed in the works \cite{La18,BBLZ}. This theory has been successfully employed in various applications, e.g.\ in mathematical finance and stochastic analysis \cite{ABZ,ABC19}, in operations research \cite{Pflug,PflugPichler,PflugPichlerbook}, and in machine learning \cite{AMX}. The novel numerical method developed in the present article can therefore be used to approach a variety of problems, including for example the value of information in dynamic control problems \cite{ABZ}, stability of superhedging and utility maximization 
\cite{BBBE19}, and McKean-Vlasov optimization problems \cite{ABC19}.

A series of numerical methods have been proposed for MFG to help gaining a better understanding of equilibrium, see e.g.\  \cite{achdou2010mean,achdou2020mean,almulla2017two,benamou2017variational,gomes2014socio}.  
On the other hand, numerics for a symmetrized version of causal optimal transport, the so-called bicausal transport, have been developed by Pflug and Pichler in \cite{Pflug,PflugPichler,PflugPichlerbook}, who refer to it as nested distance. The method used in those papers, that relies on a dynamic programming principle, seems to be ill-suited for the causal transport problem (cf.\ \cite{An17}). To the best of our knowledge, the present paper provides the first numerical method for causal optimal transport. For this, we leverage on the regularization approach developed by Cuturi \cite{cuturi2013sinkhorn} in the classical optimal transport framework; see also \cite{2018-Peyre-computational-ot}.

}

\section{The N-player game and its associated Cournot-Nash limit}\label{sect.npl}
Let $\XX$ and $\YY$ be two Polish spaces equipped with their respective sigma-algebras, and $T\in\NN$ a fixed time horizon. We consider a population of $N$ players indexed by $i\in\{1,\dots,N\}$. 

At every time $t\in\{1,\cdots,T\}$, player $i$ is characterized by its own type, which is an element of $\XX$, and is denoted by $x^{i}_t$. 
We write $$x_t^\bullet:= (x^{1}_t,\cdots,x^{N}_t) $$ for the vector collecting the types of the population of players at time $t$, and similarly $$x^i_\bullet:=(x_1^{i},\cdots,x_T^{i})$$ for the vector collecting the type-path of  player $i$. 
Accordingly, $\XX^{N\times T}$ is the space of all possible type paths. We denote by $X$ a typical element of $\XX^{N\times T}$, which we may represent as either  $$X=( x^\bullet_1,\cdots, x^\bullet_T)\,\, \text{ or }\,\,X=(x_\bullet^1,\dots,x_\bullet^N ),$$ as context will reveal. Finally, in situations where the number of players $N$ is varying, we will write $x_t^{N,i}$, $x_\bullet^{N,i}$, $x_t^{N,\bullet}$, and $X^N$, to stress the dependence on $N$. Otherwise we will drop the dependence on $N$ for most of the objects yet to be introduced.

From the beginning of the game, type distributions are known for all players, and we denote them by $\eta^{N,i}\in\Pcal(\XX^T)$. We so define the associated product measure $$\eta^N:=\motimes_{i\leq N}\eta^{N,i}\in \Pcal(\XX^{N\times T}).$$ On the other hand, at time $t$ the collection of players' time-$t$ types is publicly known: We thus set $$\Fcal_t:=\sigma( x^\bullet_s\,:\, s\leq t).$$ 
 
We now describe the actions that the players are to decide. At every time $t$, each agent $i$ needs to choose an action from $\YY$. This choice may only depend on the available information at time $t$, and is therefore defined through an $\Fcal_t$-measurable function $$y^{i}_t: (\XX^{N\times T},\Fcal_t) \to \YY.$$ We define $y_\bullet^{i}$, $y_t^{\bullet}$, and $Y$, with the same conventions as before. We henceforth call \emph{pure strategy} an element $$Y=( y^\bullet_1,\cdots, y^\bullet_T)\,\, \text{ or }\,\,Y=(y_\bullet^1,\dots,y_\bullet^N ).$$

We consider a cost function
\begin{equation}\label{eq cost}
F:\XX^T\times\YY^T\times\Pcal(\YY^T)\to \RR.
\end{equation}
This captures the fact that each player faces a cost which depends on its own type (eg.\ $x^i_\bullet$) and action (eg.\ $y^i_\bullet(X)$), and on the distribution of actions of all other players (anonymous game). 
For simplicity, we assume throughout that $F$ is measurable and lower-bounded, so all integrals that we will considered are well-defined.

From the beginning of the game, the types distributions $\eta^{N,i}, i\leq N$, are known, thus each player needs to average over all possible evolutions of types. Therefore, for every pure strategy $Y$, the cost faced by player $i$ is
\begin{align*}
J^i\left(y^{i}_\bullet,y^{-i}_\bullet \right):=\int_{\XX^{N\times T}} F\left(x^{i}_\bullet,y^{i}_\bullet(X),\textstyle{\frac{1}{N-1}\sum_{k\neq i} \delta_{y^{k}_\bullet(X)}} \right)\, \eta^N(dX),
\end{align*}
{where we use $y^{-i}_\bullet$ to denote $\{y^{k}_\bullet\}_{k\neq i}$.}

\begin{defn}[Pure Nash equilibrium]
Strategy $Y$ is a \emph{pure Nash equilibrium} for the $N$-player game if\footnote{The r.h.s.\ corresponds to replacing the action $y^i_\bullet$ by $A$ for player $i$, leaving all other players unmodified.}
\begin{align}
J^i\left(y^{i}_\bullet,y^{-i}_\bullet \right) \leq J^i\left(A,y^{-i}_\bullet \right),
\end{align}
for all $i\leq N$ and all $\Fcal$-adapted $\YY$-valued processes $A$.
\end{defn}
As pure equilibria rarely exist, a typical approach is to consider randomized strategies: the choice of action at time $t$ depends on types only up to time $t$ (``non-anticipativity"), but there may be also other factors influencing the game (eg.\ independent randomization), thus $\Fcal$-adaptedness may fail: 
A measurable function $Z:\XX^{N\times T}\to\Pcal(\YY^{N\times T})$ is called \emph{mixed strategy} if, for all $t\leq T$ and all bounded Borel functions $f:\YY^{N\times t}\to \RR$, the map
$$
\XX^{N\times T}\ni X\mapsto\int_{\YY^{N\times T}} f\left(\left\{y^{\bullet}_s \right\}_{s\leq t}\right)\ Z(X)(dY)
$$ 
is $\Fcal_t$-measurable.\footnote{In plain language: A mixed strategy assigns an `action distribution' to each type-path $X$. The $X$-dependence of such an `action distribution' is assumed to be $\Fcal$-adapted.} Note that this is equivalent to the conditional independence $ y^\bullet_t \perp  x^\bullet_u$ given $\Fcal_t$, for all $u>t$, under the measure 
\begin{align}\label{eq def M}
M(dX,dY):=Z(X)(dY)\eta^N(dX)\in\Pcal(\YY^{N\times T}\times\XX^{N\times T}).
\end{align}
The cost faced by player $i$ when following a mixed strategy $Z$ is then given by
\begin{equation}\label{eq Li}
L^i (Z):=\int_{\YY^{N\times T}\times\XX^{N\times T}} F\left(x^{i}_\bullet,y^{i}_\bullet,\textstyle{\frac{1}{N-1}\sum_{k\neq i} \delta_{y^{k}_\bullet}}\right)\, 
M(dX,dY).
\end{equation}
In what follows we write $Z(dY)$ rather than $Z(X)(dY)$, if context is clear.

\begin{defn}[Mixed Nash equilibrium]\label{def.mn}
A mixed strategy $Z$ is called a \emph{mixed Nash equilibrium} for the $N$-player game if
$$ L^i(Z)\leq L^i(\widetilde Z),$$
for all $i\leq N$ and all mixed strategies $\widetilde Z$ satisfying:
\begin{align}\label{eq mix Nash}
& f:\YY^{(N-1)\times T}\to \RR \,\,\text{ bounded measurable}\\ \notag \Longrightarrow & \int_{\YY^{N\times T}} f\left(y^{-i}_\bullet\right)  Z(X)(dY)\stackrel{\eta^N-as.}{=}\int_{\YY^{N\times T}} f\left(y^{-i}_\bullet\right)  \widetilde Z(X)(dY). 
\end{align}
\end{defn}

We note that Condition \eqref{eq mix Nash} can be rephrased as: If $M$ and $\tilde M$ are associated to $Z$ and $\tilde Z$ as in \eqref{eq def M}, then the joint distribution of $\{X,y^{-i}_\bullet\}$ is the same under either $M$ or $\tilde M$.

Finding or characterizing equilibria for the $N$-player game is an extremely hard task, even when existence can be proved. Hence rather than directly studying (mixed) Nash equilibria, we will formulate an alternative equilibrium problem for a ``representative player". Proposition~\ref{Prop eNash} below shows that this provides a way to approximate equilibria for large (but finite) systems of players. The following definition translates the non-anticipativity of mixed-strategies to the asymptotic framework where essentially we deal with a continuum of identical players.\\ 

{ \noindent {\bf Notation:} For a probability measure $\pi\in\Pcal(\XX^T\times\YY^T)$, with $\XX^T$-marginal (resp. $\YY^T$-marginal) of $\pi$ we mean the distribution of the projection of $\pi$ onto $\XX^T$ (resp. $\YY^T$). We also denote by $\pi(dy_t|x_1,\dots,x_s)$ the conditional law of the $y_t$ coordinate given the $x_1,\dots,x_s$ coordinates under $\pi$. }

\begin{defn}\label{def causal}
A probability measure $\pi\in\Pcal(\XX^T\times\YY^T)$ is called \emph{causal} if
\begin{equation}\label{eq causality}
\pi(dy_t|x_1,\cdots,x_T)\stackrel{as.}{=}\pi(dy_t|x_1,\cdots,x_t),\,\, \text{for all $t=1,\dots,T-1$}.
\end{equation}
If $\pi$ has $\XX^T$-marginal $\mu$ and $\YY^T$-marginal $\nu$, it is called a causal (Kantorovich) transport from $\mu$ into $\nu$.
Furthermore, $\pi$ is called \emph{pure (or a Monge transport)} if, for all $t$, $y_t=g_t(x_1,\cdots,x_t)$ $\pi$-a.s. for some measurable function { $g_t:\XX^t\to\YY$}.
\end{defn}
The expression in \eqref{eq causality} means that, at every time $t$, given $x_1,\cdots,x_t$ (``history of $x$ up to time $t$"), the conditional distribution of $y_t$ under $\pi$ is independent of $x_s$ for all $s>t$ (``future of $x$"). In other words, the amount of mass transported by $\pi$ into a subset of $\YY^t$ depends on the source space $\XX^T$ only up to time $t$. {In probabilistic language, a pure causal plan is the joint law of a pair of processes $(X,Y)$ where $Y$ is adapted to the information of $X$, and in a sense that can be made precise, general causal plans are combinations (mixtures) of pure plans.}

\begin{defn}[Cournot-Nash equilibrium]\label{def:CN}Given $\eta\in\Pcal(\XX^T)$, a causal transport $\widehat\pi\in\Pcal(\XX^T\times\YY^T)$ is called \emph{dynamic Cournot-Nash equilibrium} for a type-$\eta$ population if, denoting by $\nu$ the $\YY^T$-marginal of  $\widehat\pi$, it holds
\begin{equation}\label{eq:CN} 
\widehat\pi\in\argmin_\pi \int_{\XX^T\times\YY^T}F(x,y,\nu)\ \pi(dx,dy),
\end{equation}
where  minimization is done over causal transports $\pi\in\Pcal(\XX^T\times\YY^T)$ with $\XX^T$-marginal $\eta$.
\end{defn}

Note that finding dynamic Cournot-Nash equilibria amounts to solving a fixed point problem: for any measure $\nu\in\Pcal(\YY^T)$, we first need to solve the minimization problem in \eqref{eq:CN} and then check whether the solution has $\YY^T$-marginal $\nu$.

We henceforth fix a compatible metric on $\YY$, with which we define the sum-metric $D_{\YY^T}$ on $\YY^T$. For $p\geq 1$, we introduce the $p$-Wasserstein distance $\mathcal{W}_p$ on the set $\Pcal_p(\YY^T)\subseteq\Pcal(\YY^T)$ of measures which integrate $D_{\YY^T}(\cdot,z)^p$ for some (and then all) $z \in \YY^T$:
{
\[
\mathcal{W}_p(\mu,\nu)=\inf\left\{\int_{\YY^T\times\YY^T} D_{\YY^T}(y,z)^p \pi(dy,dz):\pi\in\Pcal(\YY^T\times\YY^T)\, \text{with marginals $\mu,\nu$}\ \right\}^{1/p}.
\]
}

We now proceed to present two results, Propositions \ref{Prop eNash} and \ref{Prop conv eq}, which together justify the interpretation of dynamic Cournot-Nash equilibria as an asymptotic version of dynamic Nash equilibria in finite games (Definition~\ref{def.mn}). For static games this was already explored by Blanchet and Carlier \cite{BC14a}, while in the continuous-time case this runs very close to the work of Lacker \cite{L16} for mean field games. 

\begin{prop}\label{Prop eNash}
Assume that for some $p\geq 1$ and $C\geq 0$ we have 
$$
{|F(x,y,\mu)-F(x,y,\nu)|\leq C \mathcal{W}_p(\mu,\nu)\qquad \forall\ x\in\XX^T,\ y\in\YY^T,\ \mu, \nu\in\Pcal_p(\YY^T)}.
$$ 
Let $\widehat\pi$ be a dynamic Cournot-Nash equilibrium for a type-$\eta$ population and such that its $\YY$-marginal $\nu$ belongs to $\Pcal_p(\YY^T)$. For each $N$, consider the $N$-agent problem with types $\eta^{N,i}:=\eta, i\leq N$. Then, for every $\epsilon>0$, there is $N_\epsilon\in\NN$ such that, for all $N\geq N_\epsilon$, the strategy $$Z^N(X^N){ (dY^N)}:=\motimes_{k\leq N} \widehat\pi({dy^{N,k}_\bullet}| (x_1,\dots,x_T)= x^{N,k}_\bullet) $$ is an $N$-player mixed $\epsilon$-Nash equilibrium, in the sense that
\begin{equation}\label{eq eN}
L^{N,i}(Z^N)\leq L^{N,i}(\widetilde Z^N)+\epsilon,
\end{equation}
for all $i\leq N$ and all $N$-player mixed strategies $\widetilde Z^N$ such that \eqref{eq mix Nash} holds.
\end{prop}

{A first result on the approximations of Nash equilibria via solutions of mean field games was proven by P.L.\ Lions during his lectures at Coll\`ege de France, see \cite{Car10}.}

\begin{rem}
The assumption $\eta^{N,i}=\eta,  i\leq N$, in the above proposition can be replaced by the weaker assumption that the sequence $\eta^N:=\bigotimes_{i\leq N}\eta^{N,i}, N\in\NN$, is $\eta$-chaotic: for all $k$ and all $\varphi_1,\cdots,\varphi_k:\XX^T\to\RR$ continuous and bounded,
\[
\lim_{N\to\infty}\textstyle{ \int \prod_{j\leq k}\varphi_j(X^{N,j})\eta^N(dX^N)=\prod_{j\leq k}\int\varphi_j(x) \eta(dx)}.
\]
We decided to present the result in the simpler case stated in the proposition in order to avoid technicalities and keep the proof shorter.
\end{rem}

\begin{rem}
Note that, in the $\epsilon$-Nash equilibrium proposed in Proposition \ref{Prop eNash}, each player $i$ plays a strategy which is $ x^{N,i}_\bullet$-adapted. That is, even with full information on the type-path evolution of the entire population at hand, there is an approximate Nash equilibrium where players determine their strategies according only to their own type path, and that can be constructed based on a Cournot-Nash equilibrium. In MFG terms, such ``myopic" strategies are called \emph{distributed}. 
Note that $Z^N$ also forms an approximate Nash equilibrium for any partial-information version of the game, as long as player $i$ can at least observe the evolution of its own type. 
\end{rem}

\begin{proof}[Proof of Proposition~\ref{Prop eNash}]
{Fix $i\in\{1,\ldots,N\}$ and let} $\widetilde Z^N$ be a mixed strategy satisfying {\eqref{eq mix Nash} w.r.t.\ $Z^N$}. To ease the notation, we set $\nu^{Y^{N,-i}}:=\frac{1}{N-1}\sum_{k\neq i}\delta_{y^{N,k}_\bullet}$ for all $Y^N\in\YY^{N\times T}$, and write $\eta^{\otimes N}(dX^N)$ for $\motimes_{k\leq N}\eta(dx^{N,k}_\bullet)$, and $\nu^{\otimes \infty}$ for the law of an iid sequence of $\nu$-distributed random variables. Clearly $$L^{N,i}(Z^N) = \int F\left(x^{N,i}_\bullet,y^{N,i}_\bullet,\nu^{Y^{N,-i}}\right) \motimes_{k\leq N} \widehat\pi(dx^{N,k}_\bullet,dy^{N,k}_\bullet).  $$ 
Note that the measure $\tilde \pi \in\Pcal(\XX^T\times\YY^T)$, defined as the $(x_\bullet^{N,i},y_\bullet^{N,i})$-marginal of the measure $\tilde Z^N(X^N)(dY^N)\eta^{\otimes N}(dX^N)$,
has $\XX^T$-marginal $\eta$ and satisfies \eqref{eq causality}. It follows that $L^{N,i}(Z^N)-L^{N,i}(\widetilde Z^N)$ is equal to
\begin{eqnarray*}
\int \left( F\left(x^{N,i}_\bullet,y^{N,i}_\bullet,\nu^{Y^{N,-i}}\right) - F\left(x^{N,i}_\bullet,y^{N,i}_\bullet,\nu\right) \right)\, 
\motimes_{k\leq N} \widehat\pi(dx^{N,k}_\bullet,dy^{N,k}_\bullet)\\
-\int \left(F\left(x^{N,i}_\bullet,y^{N,i}_\bullet,\nu^{Y^{N,-i}}\right) - F\left(x^{N,i}_\bullet,y^{N,i}_\bullet,\nu\right) \right)\, \tilde Z^N(X^N)(dY^N)\eta^{\otimes N}(dX^N)\\
+ \int_{\YY^T\times\XX^T}F\left(x^{N,i}_\bullet,y^{N,i}_\bullet,\nu\right)\, \left(\widehat\pi(dx^{N,i}_\bullet,dy^{N,i}_\bullet)-
\tilde\pi(dx^{N,i}_\bullet,dy^{N,i}_\bullet)\right).
\end{eqnarray*}
By definition of Cournot-Nash, the last term is non-positive, and from this we derive
\begin{eqnarray}\label{eq two}
\begin{split}
& L^{N,i}(Z^N)-L^{N,i}(\widetilde Z^N) \\  & \hspace*{1cm}\leq
\int_{\YY^{N\times T}\times\XX^{N\times T}} 
P\left(x^{N,i}_\bullet,y^{N,i}_\bullet,\nu^{Y^{N,-i}},\nu\right)\, 
\motimes_{k\leq N}\widehat\pi(dx^{N,k}_\bullet,dy^{N,k}_\bullet)
\hspace*{2.1cm}\\ & \hspace*{1.4cm}
-\int_{\YY^{N\times T}\times\XX^{N\times T}} P\left(x^{N,i}_\bullet,y^{N,i}_\bullet,\nu^{Y^{N,-i}},\nu\right)\, \widetilde Z^N(X^N)(dY^{N})\, \eta^{\otimes N}(dX^{N}),
\end{split}
\end{eqnarray}
where $P\left(x,y,\xi,\zeta\right):=F\left(x,y,\xi\right) - F\left(x,y,\zeta\right)$.
Note that, by the assumption of Lipschitz continuity of $F$, we have $|P\left(x^{N,i}_\bullet,y^{N,i}_\bullet,\nu^{Y^{N,-i}},\nu\right)|\leq C \mathcal{W}_p(\nu^{Y^{N,-i}},\nu)$, hence the r.h.s.\ of \eqref{eq two} is dominated by 
\begin{eqnarray}\label{eq 2cw}
\begin{split}
C\int_{\YY^{N\times T}\times\XX^{N\times T}} \mathcal{W}_p(\nu^{Y^{N,-i}},\nu)\, 
\motimes_{k\leq N}\widehat\pi(dx^{N,k}_\bullet,dy^{N,k}_\bullet)
\\
+C\int_{\YY^{N\times T}\times\XX^{N\times T}} \mathcal{W}_p(\nu^{Y^{N,-i}},\nu)\, \widetilde Z^N(X^N)(dY^{N})\, \eta^{\otimes N}(dX^{N})\\
=2C\int_{\YY^{N\times T}} \mathcal{W}_p(\nu^{Y^{N,-i}},\nu)\, 
\nu^{\otimes N}(dY^{N}),
\end{split}
\end{eqnarray}
{where the equality follows from the fact that $\tilde{Z}^N$ satisfies \eqref{eq mix Nash} w.r.t.\ $Z^N$.} 
On the other hand, for $y^{N,k}_\bullet$ iid $\sim\nu$,  LLN implies the following two convergences {as $N\to\infty$ for any fixed $i\leq N$}:
\[
{ R_N(z,Y^N)}:=\int D_{\YY^T}(u,z)^p \nu^{Y^{N,-i}}(du)=\frac{1}{N-1}\sum_{k\neq i}D_{\YY^T}(y^{N,k}_\bullet,z)^p\to \int D_{\YY^T}(u,z)^p \nu(du)
\]
{ $\nu^{\otimes \infty}$-a.s.}, and $\nu^{Y^{N,-i}}$ weakly converges to 
$\nu$. These convergences in turn imply 
\[
\mathcal{W}_p(\nu^{Y^{N,-i}},\nu) \to 0\quad \nu^{\otimes \infty}\text{-a.s.}
\]
Since $\int { R_N(z,Y^N) \nu^{\otimes \infty}(dY^{N})} = \int D_{\YY^T}(u,z)^p \nu(du)$ for each $N$, also ${R_N(z,Y^N)}\to \int D_{\YY^T}(u,z)^p \nu(du) $ in $L^1(\nu^{\otimes \infty})$, and so the r.h.s.\ of $$\mathcal{W}_p(\nu^{Y^{N,-i}},\nu)\leq \left (\int D_{\YY^T}(u,z)^p \nu^{Y^{N,-i}}(du)\right)^{1/p}+  \left (\int D_{\YY^T}(u,z)^p \nu(du)\right)^{1/p}$$
is uniformly integrable, and thus also the l.h.s.\ is. This yields
\[
\int \mathcal{W}_p(\nu^{Y^{N,-i}},\nu)\ 
{ \nu^{\otimes \infty}(dY^{N})}\to 0.
\]
Therefore, for every $\epsilon>0$, there is $N_\epsilon\in\NN$ such that, for all $N\geq N_\epsilon$ and all $i\leq N$,
\begin{equation*}
\int_{\YY^{N\times T}\times\XX^{N\times T}} \mathcal{W}_p(\nu^{Y^{N,-i}},\nu)\, 
\nu^{\otimes N}(dY^{N}) \leq \epsilon/2C,
\end{equation*}
and the result follows from \eqref{eq 2cw}.
\end{proof}

We now want to prove some kind of converse of Proposition~\ref{Prop eNash}, that is, to find dynamic Cournot-Nash equilibria as limit of dynamic ($\epsilon$-)Nash equilibria in $N$-player games, when the size of the population tends to infinity. This is notoriously the difficult implication, and indeed, in order to get such a result, we require a strong assumption.

\begin{prop}\label{Prop conv eq}
Let $Z^N$ be a dynamic Nash equilibrium for the $N$-player problem with types $\eta^N$, and set $\pi^N(dX^N,dY^N):=Z^N(X^N)(dY^N)\eta^N(dX^N)\in \Pcal(\XX^{N\times T}\times \YY^{N\times T})$. Assume that the cost function $F$ is continuous on $\XX^T\times\YY^T\times \Pcal_p(\YY^T)$ and bounded.
For $i\leq N$, define the (random) measures on $\YY^T$:
\[
\nu^{Y^{N,-i}}:=\frac{1}{N-1}\sum_{k\neq i}\delta_{y^{N,k}_\bullet}.
\]
Assume that, for $N\to\infty$, the sequence
\begin{equation}\label{eq Pn}
P^N:=\frac1N\sum_{i=1}^N \pi^N\circ \left(x^{N,i}_\bullet,y^{N,i}_\bullet,\nu^{Y^{N,-i}}\right)^{-1}
\end{equation}
converges to $P(dx,dy,d\xi)\in\Pcal(\XX^T\times\YY^T\times
\Pcal(\YY^T))$ in the sense that $$\int L(x,y,\xi)P^N(dx,dy,d\xi)\to \int L(x,y,\xi)P(dx,dy,d\xi) $$ for all $L$ bounded, measurable in the first argument and continuous in the last two ones. Assume further that 
 $P\circ\xi^{-1}=\delta_{\nu}$ where $\nu:=P\circ y^{-1}$.
Then $\hat\pi:=P\circ(x,y)^{-1}$ is a Cournot-Nash equilibrium for a type-$\eta$ population, where $\eta:=P\circ x^{-1}$.
\end{prop}

\begin{rem}
1. An analogue version of this proposition still holds true if, instead of Nash equilibria, the $Z^N$ are taken to be $\epsilon$-Nash equilibria in the sense of \eqref{eq eN}.\\
2. Let $\widehat\pi$ be a Cournot-Nash equilibrium with $\YY$-marginal $\nu$, and $\{Z^N\}_{N\in\NN}$ be the corresponding $\epsilon$-Nash equilibria for the $N$-player games constructed in Proposition~\ref{Prop eNash}. Then the sequence in \eqref{eq Pn} converges to $P(dx,dy,d\xi)=\widehat\pi(dx,dy) \delta_\nu(d\xi)$  in the desired sense. \\
3. { The strong assumption in this proposition is the degeneracy condition $P\circ\xi^{-1}=\delta_{P\circ y^{-1}}$.
Without it, one would need to interpret the limit $P$ as a \emph{weak} Cournot-Nash equilibrium, in the same spirit as \cite{L16}. Here we do not develop this weaker formulation.} 
\end{rem}

\begin{proof}
 We first show that the cost associated to $\hat\pi$ is smaller than the cost associated to any causal measure of pure type with $\XX^T$-marginal $\eta$.
Fix such a measure $\pi\in\Pcal(\XX^T\times\YY^T)$, with corresponding functions $g_t,t=1,\cdots,T$ as in Definition~\ref{def causal}, and set
$G(x):=(g_t(x_1,\cdots,x_t))_{t=1}^T$. 
Since $Z^N$ is Nash, for all $i\leq N$ we have
\[
L^i(Z^N)\leq L^i(Z^{N,i}),
\]
where $Z^{N,i}(X^N)(dY^N):=Z^{N}(X^N)({ dy^{N,-i}_\bullet})\delta_{G(x^{N,i}_\bullet)}(dy^{N,i}_\bullet)$.
By summing up over $i$ and dividing by $N$, we get
\begin{eqnarray}\label{eq sums}
\begin{split}
\frac1N\sum_{i\leq N}\int_{\XX^{N\times T}\times\YY^{N\times T}} F\left(x^{N,i}_\bullet,y^{N,i}_\bullet,\nu^{Y^{N,-i}}\right)\, 
\pi^N(dX^N,dY^N)\\
\leq\frac1N\sum_{i\leq N}\int_{\XX^{N\times T}\times\YY^{N\times T}} F\left(x^{N,i}_\bullet,G(x^{N,i}_\bullet),\nu^{Y^{N,-i}}\right)\, 
\pi^N(dX^N,{ dy^{N,-i}_\bullet}).
\end{split}
\end{eqnarray}
Now, by convergence of the $P^N$ in \eqref{eq Pn}, for $N\to\infty$ the l.h.s.\ in \eqref{eq sums} converges to
\[
\int_{\XX^T\times\YY^T}F(x,y,\nu)\ \hat\pi(dx,dy),
\]
whereas the r.h.s.\ converges to
\[
\int_{\XX^{T}} F\left(x,G(x),\nu\right)\, \eta(dx)=\int_{\XX^T\times\YY^T}F(x,y,\nu)\ \pi(dx,dy).
\]
Therefore
\begin{align}
\int_{\XX^T\times\YY^T}F(x,y,\nu)\ \hat\pi(dx,dy)\leq \int_{\XX^T\times\YY^T}F(x,y,\nu)\ \pi(dx,dy)\label{eq half way}
\end{align}
for every pure causal measures $\pi$ with $\XX^T$-marginal $\eta$.

 To justify that $\hat\pi$ is a Cournot-Nash equilibrium for a type-$\eta$ population, it remains to extend \eqref{eq half way} to $\pi$ not necessarily pure. From the ``chattering lemma" (see, e.g., \cite[Theorem~2.2]{ENJ88}, \cite[Theorem~4]{FN84}, \cite[Lemma~6.5]{L16}),
for any causal measure $\pi$ with $\XX^T$-marginal $\eta$, there are pure causal measures $\{\pi_n\}_{n\in\NN}$ with the same $\XX^T$-marginal and such that $\pi_n\to\pi$ weakly. The proof then follows from \eqref{eq half way}, being $F$ continuous in $x,y$ and bounded from above.
\end{proof}

The above analysis, in particular Proposition~\ref{Prop eNash} and Proposition~\ref{Prop conv eq}, justify the study of Cournot-Nash equilibria, which we do in Section~\ref{Sect OT} below.

\section{Optimal transport perspective}\label{Sect OT}
In this section we study existence, uniqueness, and characterization of dynamic Cournot-Nash equilibria by means of optimal transport techniques.

We start by recalling the general formulation of the classical optimal transport problem: given two Polish measure spaces $(\Xcal,\eta)$ and $(\Ycal,\nu)$, and a cost function $c:\Xcal\times\Ycal\to\RR$, the \emph{optimal transport problem} is defined as
\[
\inf \left\{\textstyle{\int_{\Xcal\times\Ycal} c(x,y)\ \pi(dx,dy)} : \pi\in \Pi(\eta,\nu)\right\},
\]
where 
$\Pi(\eta,\nu)=\left\{\pi\in\Pcal(\Xcal\times\Ycal)\, \text{with $\Xcal$-marginal $\eta$ and $\Ycal$-marginal $\nu$}\right\}$ is the set of all transports of $\eta$ into $\nu$.
If $\Xcal$ and $\Ycal$ are endowed with filtrations, say $(\Fcal^\Xcal_t)_{t=1,\cdots,T}$ and $(\Fcal^\Ycal_t)_{t=1,\cdots,T}$, we can also define the \emph{causal optimal transport (COT) problem} 
\begin{equation}\label{eq COT}
\inf \left\{\textstyle{\int_{\Xcal\times\Ycal} c(x,y)\ \pi(dx,dy)} : \pi\in \Pi^{\Fcal^\Xcal,\Fcal^\Ycal}_c(\eta,\nu)\right\},
\end{equation}
where $\Pi^{\Fcal^\Xcal,\Fcal^\Ycal}_c(\eta,\nu)$ denotes the subset of measures in $\Pi(\eta,\nu)$ which are causal w.r.t. $\Fcal^\Xcal,\Fcal^\Ycal$, that is, such that
$\forall t$ and $D\in\Fcal^\Ycal_t$,\; the map $\Xcal\ni x \mapsto  \pi(y\in D|x)$
is $\Fcal^{\Xcal}_t$-measurable. {
Roughly speaking, the amount of mass transported by a causal plan $\pi$ to a subset of the target space $\Ycal$ belonging
to $\Fcal^\Ycal_t$ depends on the source space $\Xcal$ only up to time $t$. Thus, a causal plan transports mass in a non-anticipative way.}

From now on we will consider transports between the spaces $\Xcal=\XX^T$ and $\Ycal=\YY^T$, endowed with the respective canonical filtrations. In this way the above definition of causality corresponds to the one we gave in Definition~\ref{def causal} above. In this setting we will simply denote by $\Pi_c(\eta,\nu)$ the set of causal transports of $\eta$ into $\nu$. Looking at \eqref{eq:CN}, it is clear that finding dynamic Cournot-Nash equilibria amounts to solving a generalized version of a causal transport problem in which the second marginal is not fixed (while the first marginal is the population type). In particular, a dynamic Cournot-Nash equilibrium $\widehat\pi$ with marginals $\eta$ and $\nu$ will also solve the COT problem in the sense of \eqref{eq COT}, between the measures $\eta$ and $\nu$ and with cost function $c(x,y)=F(x,y,\nu)$.

\begin{rem}\label{rem fp}
Finding Cournot-Nash equilibria for a type-$\eta$ population amounts to solving the following fixed point problem:
\begin{itemize}
\item[1.] $\widehat\pi\in\argmin_{\pi\in\Pi_c(\eta,.)} \int_{\XX^T\times\YY^T}F(x,y,\widehat\nu)\ \pi(dx,dy) $,\, for some $\widehat\nu\in\Pcal(\YY^T)$;
\item[2.] $\widehat\pi$ has second marginal $\widehat\nu$,
\end{itemize}
where we used the notation $\Pi_c(\eta,.)=\cup_{\xi\in\Pcal(\YY^T)}\Pi_c(\eta,\xi)$.
\end{rem}

\subsection{Potential Games}\label{sec potential}

In this section we consider a \emph{separable cost}, that is,
\begin{eqnarray}\label{eq:cost}
F(x,y,\nu)=f(x,y)+V[\nu](y).
\end{eqnarray}
This means that we are considering the case where agents face an idiosyncratic part of the cost, only depending on their own type and action, and a mean-field component depending on other agents' strategies; see the discussion after \eqref{eq cost}.

\begin{exa}\label{ex ra}
Classical examples for the mean-field interaction term consist in penalizing (resp. encouraging) similar actions among players. For example (cf. \cite{BC16}):
\begin{itemize}
\item \emph{repulsive effect (congestion)}: $V_r[\nu](y)=h\left(y,\frac{d\nu}{dm}(y)\right)$, where $\nu\ll m$, $m\in\Pcal(\YY^T)$ is a reference measure w.r.t.\ which congestion is measured, and $h(y,.)$ is increasing. This translates the fact that frequently played strategies are costly, and leads to dispersion in the strategy distribution w.r.t.\ $m$.
\item \emph{attractive effect}: $\YY$ is a convex subset of an Euclidean space and we have $V_a[\nu](y)=\int_\Ycal\phi(y,z)\nu(dz)$, where $\phi$ is a symmetric convex function which is minimal on the diagonal and increases with the distance from the diagonal. This leads to concentration of strategies. 
\end{itemize}

\end{exa}

In this section we will consider a special class of games, so-called \emph{potential games}, where the mean-field functional $V$ is the first variation of another function $\Ecal$, called energy function.
\begin{ass}\label{ass:potential}
There exists a function $\Ecal:\Pcal(\YY^T)\to\RR$ such that the functional $V$ in \eqref{eq:cost} satisfies
\begin{equation}\label{eq pot}
\lim_{\epsilon\to 0^+}
\frac{
\Ecal(\nu+\epsilon(\xi-\nu))-\Ecal(\nu)}{\epsilon}=
\int_{\YY^T} V[\nu](y)(\xi-\nu)(dy),\quad \forall\ \nu,\xi\in\Pcal(\YY^T).
\end{equation}
We use the notation $V=\nabla_{\nu}\Ecal$.
\end{ass}
For example, the repulsive and attractive functionals $V_r, V_a$ introduced above are the first variation of the following functions, respectively:
\begin{equation}\label{eq Era}
\Ecal_r(\nu)=\int_{\YY^T} H\left(y,\frac{d\nu}{dm}(y)\right)m(dy),\quad \Ecal_a(\nu)=\frac{1}{2}
\int_{\YY^T\times\YY^T}\phi(y,z)\nu(dz)\nu(dy),
\end{equation}
where $H(y,u)=\int_0^uh(y,s)ds$.\footnote{A further example of energy function is given by $\Ecal(\nu)=\sup_{a\in A}\int_{\YY^T}l(a,y)\nu(dy)$, where $A$ is a concave, compact subset of a topological space, and $l:A\times\YY^T\to\RR$ a measurable function strictly concave in the first argument. In this case ${V}[\nu](y)=l(a^*(\nu),y)$, where  $a^*(\nu)$ is the unique optimizer for $\Ecal(\nu)$.}

Under Assumption~\ref{ass:potential}, it is natural to consider the following variational problem
\begin{eqnarray}\label{eq:vp}
\inf_{\nu\in\Pcal(\YY^T)}\big\{\text{COT$(\eta,\nu)$}+\Ecal[\nu]\big\},
\end{eqnarray}
where COT$(\eta,\nu)$ is the causal optimal transport of $\eta$ into $\nu$ w.r.t. the cost function $f$:
\begin{eqnarray}\label{eq:cotf}
\text{COT$(\eta,\nu)$}:=
\inf_{\pi\in\Pi_c(\eta,\nu)}{\textstyle \int_{\XX^T\times\YY^T}f(x,y)\ \pi(dx,dy)}.
\end{eqnarray}
We say that a pair $(\pi,\nu)$ solves the variational problem \eqref{eq:vp} if $\pi$ solves the optimization in \eqref{eq:cotf}, and its $\YY^T$-marginal $\nu$ solves the one in \eqref{eq:vp}.
The following theorem states the equivalence between Cournot-Nash equilibria and first order optimality of problem \eqref{eq:vp}.

\begin{thm}\label{thm:eqv}
Let $\Ecal$ be convex. Then the following are equivalent:
\begin{itemize}
\item[(i)] $\widehat\pi$ is a dynamic Cournot-Nash equilibrium;
\item[(ii)] the pair made of $\widehat\pi$ and its $\YY^T$-marginal solves the variational problem \eqref{eq:vp}.
\end{itemize}
\end{thm}
Convexity is for example satisfied by $\Ecal_r$ in \eqref{eq Era} since $h$ is increasing in the second argument. Note however that the request of convexity can be weakened in Theorem~\ref{thm:eqv}; see Remark~\ref{rem LLmon}.

\begin{proof}
(ii)$\Rightarrow$(i): let $(\widehat\pi,\widehat\nu)$ solve \eqref{eq:vp}, then, for any $\tilde\nu\in\Pcal(\YY^T)$, $\text{COT$(\eta,\widehat\nu)$}+\Ecal[\widehat\nu]\leq\text{COT$(\eta,\tilde\nu)$}+\Ecal[\tilde\nu]$. In particular, for any $\nu\in\Pcal(\YY^T)$ and $\epsilon>0$,
\[
\Ecal[\widehat\nu]-\Ecal[(1-\epsilon)\widehat\nu+
\epsilon\nu]\leq \text{COT}(\eta,(1-\epsilon)\widehat\nu+
\epsilon\nu)-\text{COT}(\eta,\widehat\nu)\leq\epsilon\left(\text{COT}(\eta,\nu)-\text{COT}(\eta,\widehat\nu)\right),
\]
by convexity of $\nu\mapsto\text{COT}(\eta,\nu)$.
This implies
\[
\textstyle
\text{COT}(\eta,\widehat\nu)-\text{COT}(\eta,\nu)\leq \lim_{\epsilon\to 0^+}
\frac{1}{\epsilon}\big(
\Ecal[(1-\epsilon)\widehat\nu+
\epsilon\nu]-\Ecal[\widehat\nu]\big)=
\int_{\YY^T} V[\widehat\nu](y)(\nu-\widehat\nu)(dy).
\]
Therefore, for every $\pi\in \Pi_c(\eta,\nu)$,
\begin{align*}
\textstyle
\int_{\XX^T\times\YY^T}\left[f(x,y)+V[\widehat\nu](y)\right] \widehat\pi(dx,dy)&\leq \text{COT}(\eta,\nu)+\int_{\YY^T} V[\widehat\nu](y)\nu(dy)\\ & \leq \int_{\XX^T\times\YY^T}\left[f(x,y)+V[\widehat\nu](y)\right] \pi(dx,dy).
\end{align*}
Since $\nu$ was arbitrary in $\Pcal(\YY^T)$, $\widehat\pi$ is a Cournot-Nash equilibrium by Remark~\ref{rem fp}.\\
(i)$\Rightarrow$(ii): let $\widehat\pi$ be a Cournot-Nash equilibrium, and $\widehat\nu$ be its $\YY^T$-marginal.
From 1. in Remark~\ref{rem fp},
\[
\textstyle
\widehat\pi\in\argmin_{\Pi_c(\eta,\widehat\nu)} \int_{\XX^T\times\YY^T}F(x,y,\nu)\ \pi(dx,dy)=\argmin_{\Pi_c(\eta,\widehat\nu)} \int_{\XX^T\times\YY^T}f(x,y)\ \pi(dx,dy),
\]
that is, $\widehat\pi$ solves COT$(\eta,\widehat\nu)$; see also the discussion above Remark~\ref{rem fp}.
We are then left to prove that $\widehat\nu$ solves the optimization in \eqref{eq:vp}. By running backward the argument used in the first part of the proof, and using convexity of $\Ecal$, we have that
\[
\textstyle
\text{COT}(\eta,\widehat\nu)-\text{COT}(\eta,\nu)\leq \lim_{\epsilon\to 0^+}
\frac{1}{\epsilon}\big(
\Ecal[(1-\epsilon)\widehat\nu+
\epsilon\nu]-\Ecal[\widehat\nu]\big)\leq \Ecal[\nu]-\Ecal[\widehat\nu],
\] 
which concludes the proof.

\end{proof}

\begin{rem}\label{rem LLmon}
1. Convexity of $\Ecal$ has been used only for the implication ``$(i)\Rightarrow(ii)$". In fact, we proved that $\Ecal$ convex implies that $V$ is \emph{LL-monotone}, and that this in turn ensures the implication ``$(i)\Rightarrow(ii)$". The notion of LL-monotonicity, introduced by Lasry and Lions, can be written as
\begin{equation}\label{eq LLmon}
\int (V[\nu]-V[\nu'])(y)(\nu(dy)-\nu'(dy))\geq 0\quad \forall \nu,\nu'\in\Pcal(\YY^T).
\end{equation}
2. Note that we do not have any restriction on the Polish space $\XX$, thus, already for $T=1$, Theorem~\ref{thm:eqv} above generalizes Theorem~1 and Proposition~1 in \cite{BC16}, and our arguments provides an easier proof that does not use Kantorovich duality and potentials. The same is true for the existence and uniqueness results in \cite{BC16}; see Corollaries~\ref{cor unique} and \ref{cor exist} below.

\noindent 3. Variational problems with similar structure to \eqref{eq:vp} appeared already in the work \cite{ABC19} by the first two authors and Carmona, concerning a discretization scheme for (extended) mean-field control problems. Accordingly we may expect that the algorithms in Section~\ref{sect.num} below may be of relevance in that context too.
\end{rem}

\begin{cor}[Uniqueness]\label{cor unique}
If $\Ecal$ is strictly convex, then all Cournot-Nash equilibria have the same second marginal, that is, the distribution of the optimal strategies is unique.
\end{cor}
\begin{proof}
Assume $\Ecal$ is strictly convex. Then, since $\nu\mapsto$COT$(\eta,\nu)$ is convex, the variational problem \eqref{eq:vp} admits at most one solution $\nu$. The statement then follows from Theorem~\ref{thm:eqv}.
\end{proof}
Strict convexity is for example satisfied by $\Ecal_r$ in \eqref{eq Era} when $h$ is strictly increasing in the second variable.

\begin{cor}[Existence]\label{cor exist}
Assume $f$ bounded from below and lower-semi-continuous, and consider the mean-field functionals in Example~\ref{ex ra}. Cournot-Nash equilibria exist under either:\\
1. $V=V_r$ and $h$ satisfy the growth condition: there is a coercive\footnote{Ie. $\lim_{|u|\to\infty}\ell(u)/|u|=+\infty$.} differentiable function $\ell$ such that $p(c+\ell'(u))\leq h(y,u)$, for some $p>0,c\in\mathbb{R}$, and all $u$. \\
2. $V=V_a$ , $\phi\geq 0$ and continuous, and $f(x,y)\geq \kappa\|x-y\|^p$ for some $p\geq 1$ and $\kappa >0$.
\end{cor}
\begin{proof}
Since $f$ is bounded from below and lower-semi-continuous, then COT$(\eta,\nu)$ admits an optimal solution for any pair of marginals; this is a simple consequence of Lemma \ref{lem compactness}. By the same lemma we derive that the function $\nu\mapsto \text{COT}(\eta,\nu)$ is lower semicontinuous (and of course lower bounded). Indeed, if $\nu_n\to \nu_\infty$ weakly, then $B:=\bigcup_{n\in\mathbb{N}\cup\{\infty\}} \{\nu_n\}$ is weakly compact and so $\bigcup_{\nu\in B}\Pi_c(\eta,\nu)$ is also weakly compact by Lemma \ref{lem compactness}. If $\pi_n$ attains $\text{COT}(\eta,\nu_n)$, then (up to passing to a subsequence) $\pi_n\to\pi$ for some $\pi\in \Pi_c(\eta,\nu_\infty) $, and thus $\liminf \text{COT}(\eta,\nu_n) \geq \int fd\pi\geq \text{COT}(\eta,\nu_\infty)$.   
We are then left to prove that the respective growth conditions ensure existence of solutions to the minimization in \eqref{eq:vp}, so that we can conclude by Theorem~\ref{thm:eqv}. 

For Point 1, we first derive $a+bv+p\ell(v)\leq H(y,v)$. For any optimizing sequence $\{\nu_n\}$ for \eqref{eq:vp}, we may assume that $\int\ell\left( \frac{d\nu_n}{dm}(y)\right)m(dy)\leq C$. By de la Vall\'ee-Poussin theorem, the set of densities $\{d\nu_n/dm\}_n$ is uniformly integrable and hence weakly relatively compact in $L^1(m)$. If $Z$ is any accumulation point in this set, then we have $\nu_n\to Zdm=:\nu$ weakly.  Since $H(y,\cdot)$ is convex and continuous, $\Ecal_r$ is lower semicontinuous. This shows, by the first paragraph, that $\nu$ is an optimizer of \eqref{eq:vp}. 

On the other hand, for Point 2, we first observe that $\int fd\pi\geq c_1+c_2\int \|y\|^p\nu(dy)$ for each $\pi\in \Pi_c(\eta,\nu)$, for some constants $c_1,c_2$. As a result, if $\{\nu_n\}$ is a minimizing sequence then $\int \|y\|^p\nu_n(dy)\leq C$ for some large enough constant $C$. By Markov's inequality we obtain tightness of $\{\nu_n\}$, and we may conclude as for Point 1 after proving that $\Ecal_a$ is weakly lower semicontinuous. To prove the latter, assume that $\nu_n\to \nu_\infty$ weakly. By Skorokhod representation, on some probability space there are random variables $\{W_i:i\in\mathbb{N}\cup\{\infty\}\}$ such that $W_n\to W_\infty$ a.s.\ and $W_i\sim\nu_i$ for each $i\in\mathbb{N}\cup\{\infty\}$. Possibly extending the probability space, we build an independent family of random variables $\{\tilde W_i:i\in\mathbb{N}\cup\{\infty\}\}$ with the same properties. We conclude by Fatou's lemma:
\begin{align*}
 \liminf_n\int \phi(x,y)\nu_n(dx)\nu_n(dy)&=\liminf_n E[\phi(W_n,\tilde W_n)]\geq \mathbb E[\liminf_n \phi(W_n,\tilde W_n)] \\
 &\geq \mathbb E[ \phi(W_\infty,\tilde W_\infty)] =\int \phi(x,y)\nu_\infty (dx)\nu_\infty(dy).
\end{align*}
\end{proof}

\begin{rem}\label{rem:subpopu}
In order to make a parallel between the current framework and the typical setting in mean field games, note that here there are no state dynamics, and that the cost is set as the average over possible evolution of types rather than an expectation or average ``over noise". Remarkably, the different type paths of agents could correspond to different subpopulations of players in the game, and each agent faces a cost that depends on their own type/population. Therefore our setting accommodates (possibly infinite) multiple populations. However, unlike what is generally done in the literature on multi-population mean field games (see e.g.\ \cite{ABC17}), we do aggregate all actions of the game into a single (empirical) measure rather than letting each subpopulation contribute in a separate way. Separating the actions of players according to the population they belong to would involve a different mathematical analysis. For example, in the limiting problem, the dependence on the distribution of actions of all players (what we denote by $\nu$) would be replaced by the disintegration of $\pi$ w.r.t. its first marginal $\eta$, yielding a cost $F(x,y,\bar x\mapsto\pi^{\bar x})$. 
Moreover, a separable cost in this case would take a form of the kind $f(x,y)+\int V^{\bar x}[\pi^{\bar x}](y)\eta(d\bar x)$. Of course things would simplify by considering finitely many types of agents, that is a finite numbers of subpopulations, which corresponds to having an atomic measure $\eta$. We do not pursue this analysis here.
\end{rem}

\subsection{Social planner and Price of Anarchy}
From a social planner perspective, optimal strategies in an $N$-player game are those that minimize the total cost, which corresponds to minimizing the average cost (this latter form is the suitable one in order to study this problem asymptotically). In this way we find the so called \emph{cooperative equilibria}. The corresponding optimization problem for $N\to\infty$ is the following:
\[\inf_{\pi\in\Pi_c(\eta,.)} \int_{\XX^T\times\YY^T}F(x,y,{p_2}_\#\pi)\ \pi(dx,dy),
\]
where ${p_2}_\#\pi$ denotes the second marginal of $\pi$; to be compared with the asymptotic formulation of the Nash (competitive) equilibria in Definition~\ref{def:CN} or Remark~\ref{rem fp}.
In the separable case \eqref{eq:cost}, this can be written as
\begin{eqnarray}\label{eq.coop}
\inf_{\nu\in\Pcal(\YY^T)}\left\{\text{COT$(\eta,\nu)$}+\int_{\YY^T} V[\nu](y)\ \nu(dy)\right\},
\end{eqnarray}
to be compared with the variational problem \eqref{eq:vp} obtained in the competitive case.

With the same arguments used in Corollary~\ref{cor exist}, we can prove existence of solutions for the asymptotic formulation of the cooperative problem.
\begin{cor}\label{cor exist coop}
Assume $f$ bounded from below and lower-semi-continuous, and consider the mean-field functionals in Example~\ref{ex ra}. Solutions to \eqref{eq.coop} exist under either:\\
1. $V=V_r$ and $h$ satisfy the growth condition: there is a coercive function $\ell$ such that $p(c+\ell(u))\leq h(y,u)$, for some $p>0,c\in\mathbb{R}$, and all $u$. \\
2. $V=V_a$ , $\phi\geq 0$ and continuous, and $f(x,y)\geq \kappa\|x-y\|^p$ for some $p\geq 1$ and $\kappa >0$.
\end{cor}

The \emph{Price of Anarchy} is then defined as the ratio between the worst-case Nash equilibrium total cost and the socially optimal total cost.
Asymptotically,
\begin{equation}\label{eq.poa}
\text{PoA}:=\frac{\sup_{\nu\in\mathcal{N}}TC[\nu]}{\inf_{\nu\in\Pcal(\YY^T)}TC[\nu]},
\end{equation}
where $TC[\nu]=\text{COT}(\eta,\nu)+\int V[\nu]d\nu$ is the total cost associated to the strategy distribution $\nu$, and $\mathcal{N}$ is the set of optimizers in the minimization problem in \eqref{eq:vp}.
{In Section~\ref{sect.num}, we will develop numerics to compute the Price of Anarchy, and we will illustrate its behaviour with some example.}

\section{Numerical methods} \label{sect.num}

In Section~\ref{sect.num.cot} we present a numerical scheme for the causal optimal transport problem. The algorithm uses duality to transform the problem, and then applies the Sinkhorn method to solve an essential part of it. Relying on it, we also develop a numerical method for Cournot-Nash equilibria, in Section~\ref{sect.num.cn}. The Python code implementing these algorithms is freely available and can be found in https://github.com/JunchaoJia-LSE/CNGonCOT. Numerical experiments show that the proposed methods are efficient and stable.

\subsection{A numerical method for causal optimal transport}\label{sect.num.cot}
We borrow the setting of Section~\ref{sec potential}. Thus COT$(\eta,\nu)$ denotes the causal optimal transport problem of $\eta$ into $\nu$ w.r.t.\ the cost function $f$, namely
\begin{eqnarray*}
\text{COT$(\eta,\nu)$}=
\inf_{\pi\in\Pi_c(\eta,\nu)}{\textstyle \int_{\XX^T\times\YY^T}f(x,y)\ \pi(dx,dy)}.
\end{eqnarray*}

We first recall a known result (see \cite[Proposition 2.4]{BBLZ} or \cite[Lemma 5.4]{ABZ}). As a matter of terminology, whenever $\{H_t\}_{t=1}^T$ is a stochastic process defined on $\XX^T\times\YY^T$, we shall say that ``$H$ is $x$-adapted'' if for each $t$ the r.v.\ $H_t$ is a function of $x_1,\dots,x_t$, and define ``$H$ is $y$-adapted'' in an analogue fashion. We use the notation $C_b(\mathbb{W})$ for the space of continuous bounded functions on the space $\mathbb{W}$. Finally, we say that a process $\{H_t\}_{t=1}^T$ is continuous if each r.v.\ $H_t(\cdot)$ is a continuous function.

\begin{prop} \label{equivalence-proposition}
    For $\pi \in \Pi(\eta,\nu)$, the following are equivalent:

\begin{enumerate}
\def\labelenumi{\arabic{enumi}.}
\item
  $\pi$ is causal (ie.\ $\pi \in \Pi_c(\eta,\nu)$);
\item For every bounded $y$-adapted continuous process $\{h_t\}_{t=1}^T$, and each bounded $x$-adapted $\eta$-martingale $\{M_t\}_{t=1}^T$,
we have:
\begin{equation} \label{dual-constraint}
\int_{\XX^T\times\YY^T}\left[\sum_{t<T} {h_t(y) (M_{t+1}(x)-M_t(x))}\right]\pi(dx,dy)= 0;
\end{equation}
 \item For each $t\in\{1,\dots,T\}$, $h\in C_b(\YY^t)$, and $g\in C_b(\XX^T)$, the following vanishes
\begin{align} \label{dual-constraint-original}
\hspace*{-0.3cm}\int_{\XX^T\times\YY^T}h(y)\left[g(x)-\int_{\XX^{T-t}}g(x_1,\dots,x_t,\bar x_{t+1},\dots,\bar x_T)\eta^{x_1,\dots,x_t}(d\bar x_{t+1},\dots,d\bar x_T)\right]\pi(dx,dy). 
 \end{align}
\end{enumerate}
\end{prop}

\begin{defn}\label{S-definition}
We denote by $\SS$ the linear space spanned by the \emph{stochastic integrals} appearing in Equation \eqref{dual-constraint}, namely 
\begin{align*}
\SS := \text{span} &\Bigl\{\textstyle\sum_{t<T} {h_t (M_{t+1}-M_t)}\,|\, M \text{ is a bounded $x$-adapted $\eta$-martingale, } \Bigr .  \\
&\Bigl .\quad \quad\quad\quad\quad\quad\quad\quad\quad h \text{ is bounded, $y$-adapted and continuous}   \Bigr\}.
\end{align*}
Similarly, we introduce the set $\FF$ equal to the linear span of the functions
\begin{align*}
 h(y)\left[g(x)-\int_{\XX^{T-t}}g(x_1,\dots,x_t,\bar x_{t+1},\dots,\bar x_T)\eta^{x_1,\dots,x_t}(d\bar x_{t+1},\dots,d\bar x_T)\right] ,
\end{align*}
as we vary $t\leq T,h\in C_b(\YY^t),g\in C_b(\XX^T)$.
\end{defn}

Proposition \ref{equivalence-proposition} leads to existence and duality for COT$(\eta,\nu)$. Notice that, unlike in \cite{BBLZ,ABZ}, we shall make no assumption on $\eta$ whatsoever. Hence the next result needs a proof, which we provide in the appendix.

\begin{thm}\label{thm_ex_dua}
Assume that $f$ is lower semicontinuous and bounded from below. Then COT$(\eta,\nu)$ is attained and duality holds:
\begin{align*}
\text{COT}(\eta,\nu)&= \sup_{S\in\SS}\inf_{\pi\in \Pi(\eta,\nu)} \int(f+S)d\pi = \sup_{\substack{\phi(x)+\psi(y)+S(x,y)\leq f(x,y),\\ \phi\in C_b(\XX^T),\psi\in C_b(\YY^T),S\in\SS }} \int_{\XX^T}\phi d\eta + \int_{\YY^T}\psi d\nu \\
&=\sup_{F\in\FF}\inf_{\pi\in \Pi(\eta,\nu)} \int(f+F)d\pi = 
\sup_{\substack{\phi(x)+\psi(y)+F(x,y)\leq f(x,y),\\ \phi\in C_b(\XX^T),\psi\in C_b(\YY^T),F\in\FF }} \int_{\XX^T}\phi d\eta + \int_{\YY^T}\psi d\nu .
\end{align*}
\end{thm}

Of all the above duality identities, we now solely exploit 
\begin{align}\label{eq_dual}
\text{COT}(\eta,\nu)&= \sup_{S\in\SS}\inf_{\pi\in \Pi(\eta,\nu)} \int(f+S)d\pi.
\end{align}
The driving idea is that $\inf_{\pi\in \Pi(\eta,\nu)} \int(f+S)d\pi$ can be efficiently approximated using Sinkhorn iterations. It is convenient to define
\begin{align}\label{eq_duality_use}
\text{OT}^{S}(\eta,\nu) := \inf_{\pi \in \Pi(\eta, \nu)} \int (f + S)\, d \pi.
\end{align}

From now on, we specialize to a discrete setting. To be consistent with the notation above, however, we keep the notation of integral w.r.t.\ $\pi$.

\begin{ass}\label{ass.fin}
The sets $\XX$ and $\YY$ are finite, say {$|\XX|=n$ and $|\YY|=m$}.
\end{ass}
For $\pi\in\Pcal(\XX^T\times\YY^T)$, we use the notation {$\pi_{i,j}=\pi(^ix,\,^jy)$, where $\{^ix, i=1,\ldots,n^T\}$ and $\{^jy, j=1,\ldots,m^T\}$ denote the elements in $\XX^T$ and $\YY^T$, respectively}. The entropy of $\pi$ is then defined as
$$
\text{Ent}(\pi)= -\sum_{i,k} \pi_{i,k}\log(\pi_{i,k}).
$$
Note that the entropy is uniformly bounded:
\begin{equation}\label{eq_bdd_ent}
0\leq \text{Ent}(\pi) \leq C:={|\XX^T| |\YY^T| e^{-1}=n^Tm^T e^{-1}}. 
\end{equation}

We now add an entropic penalization to the transport problem and consider, for $S\in\SS$, the problem
\begin{align}\label{eq_duality_use_eps}
\text{OT}^{\epsilon,S} (\eta,\nu):= \inf_{\pi \in \Pi(\eta, \nu)}\left\{ \int (f + S)\, d \pi\, - \epsilon \,\text{Ent}(\pi)\right\}.
\end{align}
The penalization term makes the function inside the brackets strictly convex in $\pi$, which ensures existence of a unique optimizer in the weakly compact set $\Pi(\eta, \nu)$. { The entropic regularization technique has been widely used to approximate classical optimal transport problems and easily obtain numerical solutions; see \cite{BCN18} for an application to Cournot-Nash equilibria.
}

\begin{thm}\label{thm.cotes} We have 
$$\text{COT}(\eta,\nu)=\lim_{\epsilon\searrow 0} \,\sup_{S\in\SS}\, \text{OT}^{\epsilon,S}(\eta,\nu), $$
and the convergence is monotonically increasing.
\end{thm} 

\begin{proof}
By \eqref{eq_bdd_ent}, we have
\begin{align}\label{eq_three_inequalities} \inf_{\pi\in\Pi_c(\eta,\nu)} {\textstyle \int f\ d\pi} - \epsilon\, C \leq \inf_{\pi\in\Pi_c(\eta,\nu)} \left\{{\textstyle \int f\ d\pi} - \epsilon\, \text{Ent}(\pi) \right\}\leq \inf_{\pi\in\Pi_c(\eta,\nu)}{\textstyle \int f\ d\pi}, 
\end{align}
and we conclude by noticing that 
\begin{equation}\label{eq.coteps}
 \inf_{\pi\in\Pi_c(\eta,\nu)} \left\{{\textstyle \int f\ d\pi} - \epsilon\, \text{Ent}(\pi) \right\} = \sup_{S\in\SS}\, \text{OT}^{\epsilon,S} (\eta,\nu) ,
 \end{equation}
as we justify in Lemma \ref{non-linear-duality} in the appendix.
\end{proof}

{
\begin{rem} Note that, by  \eqref{eq_three_inequalities} and \eqref{eq.coteps}, 
\[
0\leq \text{COT}(\eta,\nu)-\sup_{S\in\SS}\, \text{OT}^{\epsilon,S}(\eta,\nu) \leq \epsilon C,
\]
thus the convergence in Theorem~\ref{thm.cotes} is at least linear.
\end{rem}
}

Note that, due to Assumption~\ref{ass.fin}, the space $\SS$ is generated by a finite basis, say $\{e_k\}_{k=1}^K$ for some $K\in\NN$ (see Appendix~\ref{app.basis} for details). Then, by setting $e:=(e_1,\ldots,e_K)$, and for $\epsilon$ small enough, Theorem~\ref{thm.cotes} suggests the following approximation of the causal transport problem:
\begin{equation}\label{eq.cotapp}
\text{COT}(\eta,\nu)\approx \sup_{\lambda \in \RR^K} \inf_{\pi \in \Pi(\eta, \nu)}\left\{ \int (f+\lambda \cdot e)\, d\pi - \epsilon \text{Ent}(\pi)\right\}.
\end{equation}
Recalling the comment after \eqref{eq_duality_use_eps}, for any fixed $\lambda$ the minimization problem in \eqref{eq.cotapp} admits a unique optimizer, say $\pi^*(\lambda)$. This problem can be numerically solved in an efficient way by implementing the powerful Sinkhorn algorithm as described by Cuturi~\cite{cuturi2013sinkhorn}. On the other hand, to handle the maximization problem in \eqref{eq.cotapp}, note that Danskin's theorem \cite{danskin1966} implies the following  result, which we employ in Algorithm~\ref{algo:cot_gd} to perform gradient descent.

\begin{lem}
$\text{OT}^{\epsilon, \lambda\cdot e}(\eta,\nu)$ is differentiable w.r.t $\lambda_k \,\,  \forall k=1,\ldots, K$, with
\begin{align}\label{eq.otld}
\frac{\partial \text{OT}^{\epsilon, \lambda\cdot e}(\eta,\nu)}{\partial \lambda_k} = \int e_k \, d \pi^*(\lambda).
\end{align}
\end{lem}

The pseudo-algorithm to calculate the approximation of $\text{COT}(\eta,\nu)$ in \eqref{eq.cotapp} is displayed in Algorithm \ref{algo:cot_gd}.

\begin{algorithm}
\caption{COT by gradient descent}\label{algo:cot_gd}
\begin{algorithmic}[1]
\Procedure{COT($\eta$, $\nu$, $\epsilon$, $e$, $\delta$)}{}
\State $e \gets \text{any basis of }\SS$
\State $\lambda  \gets$ random position in $\RR^K$

\Do
    \State $\pi^* \gets \text{Sinkhorn} (\eta,\nu, f \,+ \lambda \cdot e, \epsilon)$
    \State $\lambda_k \gets \lambda_k + \delta \int e_k \, d \pi^* $
\doWhile{not (stop criterion)} 

\State \textbf{output} $\pi^*$.
\EndProcedure
\end{algorithmic}
\end{algorithm}

$\text{Sinkhorn} (\eta,\nu,f \,+ \lambda \cdot e, \epsilon)$ is the standard Sinkhorn algorithm to compute $\text{OT}^{\epsilon, \lambda \cdot e}(\eta,\nu)$ (see \cite{cuturi2013sinkhorn});  $\delta$ is a chosen gradient descent step; the stop criterion consists in the improvement of the optimal value being less than a fixed threshold. 

Algorithm~\ref{algo:cot_gd} shows the most vanilla version of a gradient descent method that we can use. Given that we have an explicit expression for the gradient, cf.\ \eqref{eq.otld}, we can also leverage any other first order optimization algorithm to perform Steps 4 to 7 in the algorithm, that is, to minimize the function $\text{Sinkhorn} (\eta,\nu,f \,+ \lambda \cdot e, \epsilon)$ in the $\lambda$ variable. Indeed, in practice we implemented such an optimization procedure using the SLSQP method in Python's Scipy optimization package, feeding the method with the explicit calculated gradients. Numerical experiments demonstrate that this outperforms, in terms of efficiency, both the vanilla gradient descent method and alternative zero order methods.

\subsection{A numerical method for dynamic Cournot-Nash equilibria}\label{sect.num.cn}
In this section we provide an algorithm to compute dynamic Cournot-Nash equilibria in the potential game setting of Section~\ref{sec potential}. Thanks to Theorem~\ref{thm:eqv}, when the energy function $\Ecal$ is convex, this boils down to solving the variational problem \eqref{eq:vp}. That is, we have an additional optimization step (in $\nu$) w.r.t.\ the previous section. 
Here again, we use Theorem~\ref{thm.cotes} to approximate causal transport problems with regularized problems $\text{OT}^{\epsilon,S}(\eta,\nu)$ given in \eqref{eq_duality_use_eps}. For $\epsilon$ small enough we obtain a reasonable approximation.

\begin{lem}
\begin{eqnarray}
\inf_{\nu\in\Pcal(\YY^T)}\big\{\text{COT$(\eta,\nu)$}+\Ecal[\nu]\big\} = \lim_{\epsilon \searrow 0} \inf_{\nu\in\Pcal(\YY^T)}\sup_{S\in \SS}\big\{\text{OT}^{\epsilon,S}(\eta,\nu)+\Ecal[\nu]\big\}.
\end{eqnarray}
\end{lem}

\begin{proof}
As in \eqref{eq_three_inequalities}, but adding $\Ecal[\nu]$ in each of its terms and then minimizing in $\nu$.
\end{proof}

As a matter of fact, in order to deal with the optimization in $\nu$ in \eqref{eq:vp}, it is convenient to use the dual formulation of the penalized transport problems (see \cite[Proposition~4.4]{2018-Peyre-computational-ot}):
\begin{equation*}\label{eq.cotsd} 
\text{OT}^{\epsilon,S}(\eta,\nu)=\sup_{\phi \in \RR^{|\XX^T|}, \psi \in \RR^{|\YY^T|}} G(\nu,\phi,\psi,\epsilon,S),
\end{equation*}
where
\[
G(\nu,\phi,\psi,\epsilon,S):=
\int \phi \,d\eta + \int \psi \,d\nu - \epsilon \exp\left(\frac{\phi}{\epsilon}\right)^{T} 
\exp\left(- \frac{f + S}{\epsilon}\right) \exp\left(\frac{\psi}{\epsilon}\right),
\]
and the exponential $\exp(\cdot)$ is understood component-wise. (To lighten the notation, we do not stress dependence on $\eta$, since this is the type distribution which is fixed.)

From now on we think of $\epsilon$ small being fixed, and set
\begin{equation}\label{eq.L}
L(\nu) := \sup_{S \in \SS, \, \phi \in \RR^{|\XX^T|}, \psi \in \RR^{|\YY^T|}}  G(\nu,\phi,\psi,\epsilon,S) + \Ecal[\nu] .
\end{equation}

\begin{prop}
The function $L$ is differentiable, with
\begin{align}\label{eq_first_var_L}
\lim_{r\searrow 0}\frac{L(\nu+r(\mu-\nu)) - L(\nu)}{r} = \int[\psi^{*}[\nu](y) + V[\nu](y)]d(\mu-\nu)(y),
\end{align}
{where $\psi^{*}[\nu]$ is any optimizer in \eqref{eq.L}.}
Without loss of generality we may assume that $\sum_y\psi^{*}[\nu](y) + V[\nu](y)=0$.
\end{prop}

\begin{proof}
Note that, for $\nu$ fixed, the supremum of $G$ in the variables $(S,\phi,\psi)$ is uniquely attained at some $S^*\in\SS$. On the other hand, there is a unique function $R$ such that $(\phi^*,\psi^*)$ are optimal if and only if $\phi^*(x)+\psi^*(y)=R(x,y)$. Then, from
Danskin's theorem \cite{danskin1966} and the definition of $V$, we have 
\begin{align*}
\lim_{r\searrow 0}\frac{L(\nu+r(\mu-\nu)) - L(\nu)}{r} = \sup_{\substack{(S^*,\phi^*,\psi^*)\\ \text{attain \eqref{eq.L}} }}\int[\psi^{*}(y) + V[\nu](y)]d(\mu-\nu)(y).
\end{align*}
Now, if $(\phi^*,\psi^*)$ and $(\tilde\phi^*,\tilde\psi^*)$ are optimizers, then $\phi^*(x)+\psi^*(y)=R(x,y)=\tilde\phi^*(x)+\tilde\psi^*(y)$, and so $\psi^*=\tilde\psi^*+c$. This establishes \eqref{eq_first_var_L}. The final conclusion follows after a possible translation of $\psi^*[\nu]$ by a constant.
\end{proof}

\begin{rem}
Note that $\psi^{*}[\nu]$ is a result of the Sinkhorn algorithm, and this facilitates the  optimization w.r.t. $\nu$. Thanks to the explicit expression of the gradient, we can search the optimal $\nu$ efficiently. This enables us to compute, in a reasonable amount of time, numerous different cases in the example of next subsection. 
\end{rem}

\begin{algorithm}
\caption{Dynamic Cournot Nash equilibria}\label{Cournot Nash}
\begin{algorithmic}[1]
\Procedure{Cournot Nash($\eta$, $f$, $\epsilon$, $\delta$, $V$)}{}
\State initialize $\nu$
\Do
    \State $\pi^*, \psi^{*} \gets \text{COT}(\eta,\nu, f, \epsilon, \delta)$
    \State $\nu \gets \nu - \delta (\psi^{*} + V[\nu])$
\doWhile{not (stop criterion)} 

\State \textbf{output} $\nu^*$
\EndProcedure
\end{algorithmic}
\end{algorithm}

All this leads to Algorithm~\ref{Cournot Nash}, in which we perform gradient decent at the level of $\nu$. To wit, Step $5$ therein uses \eqref{eq_first_var_L} to perform the descent step. (To be precise, Step 5 in fact computes first $\tilde \nu(y):=\{\nu(y) - \delta (\psi^{*}(y) + V[\nu](y)\}\vee 0$ and then outputs $\tilde \nu$ normalized so as to be a probability vector.) On the other hand, Step $4$ applies Algorithm~\ref{algo:cot_gd}. The stop criterion consists again in the improvement of the optimal value being less than a fixed threshold.

Note that the total social cost \eqref{eq.coop} can also be optimized using a simple modification of Algorithm~\ref{Cournot Nash}, where $V[\nu] (=\nabla_{\nu} \Ecal(\nu))$ is replaced by $\nabla_{\nu} (\int V[\nu](y) \nu(dy))(z)=  \int  \nabla_{\nu} (V[\nu](y))(z) \nu(dy) + V[\nu](z)$.

\section{Examples}

\subsection{A case study in a toy model}\label{sect.ex}

In this section, we apply the numerical scheme developed in the previous section in a simple two-stage game ($T=2$), cf. Example~\ref{ex_intro}. We consider a number of messengers (agents) delivering parcels from $s_0$ to
$s_1$ and then from $s_1$ to $s_2$, where 
$s_0, s_1, s_2$ are
names of three sites (Fig. \ref{fig:cities}). 

\begin{figure}[H]
\centering
\includegraphics[scale=0.7]{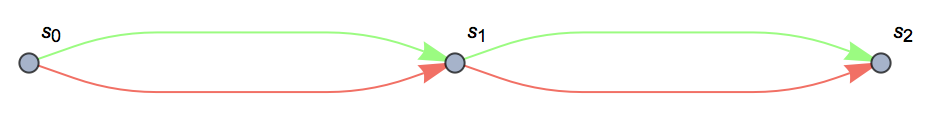}
\caption{Three sites sequentially connected by two kinds of roads}
\label{fig:cities}
\end{figure}

The structure of the example, though very simple, can be adapted to many other situations. 
For instance, we can imagine that $s_0, s_1, s_2$ are different stages of economic development. And the agents are different companies. For each company, at each stage, the market demand for its product may be either Extensive (E) or Normal (N), and the company can choose either to Quickly (Q) expand its producing capacity or to Slowly (S) do so. Naturally, different choices will result in different benefits/costs, e.g.\ if the demand is Extensive, then Quickly expanding will be better than Slowly expanding. The choice at current stage will also influence the benefits of next stage, since more capacity at this stage will also persist to next stage, which will be beneficial if the demand at that
time is Extensive.
Also, in this setting, there is a mean field congestion cost: if too many companies choose
to quickly expand at a certain stage, then resource will be scarce, and the cost will be
higher (e.g.\ higher employee wages).
Other than the specific form of cost function, this application is essentially equivalent to the messenger example proposed above. By solving this problem, we can understand different capacity expansion behaviors for different companies and also how much would be saved if the companies were regulated.

\subsubsection{Path space of types}\label{path-space-of-types}
At the starting site $s_0$, each agent gets a parcel of two possible
types:

\begin{enumerate}
\def\labelenumi{\arabic{enumi}.}
\item[E.] Express type, which is better to be delivered as quickly as possible, otherwise a penalty will be imposed.
\item[N.] Normal type, which can be delivered either quickly or slowly, and no penalty will be imposed based on that.
\end{enumerate}

After delivering the parcel to $s_1$, each agent gets another parcel to be delivered from $s_1$ to $s_2$. The parcel can again be of the two types
described above. 
Therefore, the type space is
$\XX:= \{\text{E}, \text{N}\}$, and the path space of types is $\XX^T=\{\text{E}, \text{N}\}^2$. In accordance with the notation introduced in Section~\ref{sect.npl}, we denote by $x=(x_1,x_2)$ the random process of types on the path space, that has fixed (and known) distribution $\eta$. 

\subsubsection{Path space of actions and cost function}\label{path-space-of-actions}
For delivering parcels between each pair of sites, the agents have two kinds of roads between which to choose: Quick road (Q) and Slow road (S). Therefore, the action space is $\YY:= \{\text{Q}, \text{S}\}$, and the path space of actions is $\YY^T=\{\text{Q}, \text{S}\}^2$. Taking a Slow road is penalized when delivering a parcel of Express type. Specifically, if a messenger's parcel is of Express type and she takes the Slow road, there will be a penalty of $0.5$. Taking the Slow road in the first stage also affects the delivery time of the parcel to be delivered in the second stage. This results in a cost of $0.25$ when taking S in the first stage while having a parcel E in the second stage. In addition, there is a mean-field cost that takes into account congestion, and equals the percentage of agents taking the same road.

The distribution $\nu$ of the action process has support on the four points $\{(Q,Q), (Q,S)$, $(S,Q), (S,S)\}$, and we denote the respective probabilities by $\{\nu_1, \nu_2, \nu_3, \nu_4\}$. 
The cost just described fits into the separable cost framework \eqref{eq:cost}, where the type-action part $f(x,y)$ is described by the matrix:
\begin{center}
\begin{tabular}{ p{0.5cm} | p{0.7cm} p{0.7cm} p{0.7cm} p{0.7cm}}
 & Q,Q  & Q,S & S,Q & S,S\\
 \hline
 E,E & 0 & 0.5 & 0.5 &  1\\
 E,N & 0 & 0 &  0.5 & 0.5\\
 N,E & 0 & 0.5 &  0.25 & 0.75\\
 N,N & 0 & 0 &  0 & 0
\end{tabular}
\end{center}
and the mean-field part of the cost is given by:
\begin{align*}
 V[\nu](y)  := & (2 \nu_1 + \nu_2 + \nu_3) 1_{y = (Q,Q)}  + (\nu_1 + 2\nu_2 + \nu_4) 1_{y = (Q,S)} \\ & +(\nu_1 +  2\nu_3 + \nu_4) 1_{y = (S,Q)}+(\nu_2 +  \nu_3 + 2\nu_4) 1_{y = (S,S)}.
\end{align*}
Note that the functional $V$ defined above is the first variation of the energy functional $\Ecal:\Pcal(\YY^T)\to\R$ given by
\begin{align}\label{eq.e.ex}
\Ecal(\nu) := \nu_1^2 +\nu_2^2 +\nu_3^2 +\nu_4^2+ (\nu_1+\nu_4)(\nu_2+\nu_3).
\end{align}
To see this, fix any measure $\xi\in\Pcal(\YY^T)$ and let $\zeta:=\xi-\nu$. The probabilities $\zeta_i, i=1,\cdots,4$, are defined coherently with previous notation. Then we have
\begin{align*}
& \hspace*{-0.5cm} \lim_{\epsilon\to 0^+}
\frac{\Ecal(\nu+\epsilon \, \zeta)-\Ecal(\nu)}{\epsilon} \\ & \hspace*{0.5cm}= \zeta_1 (2\nu_1 + \nu_2 + \nu_3)+\zeta_2 (\nu_1 + 2\nu_2 + \nu_4)+\zeta_3(\nu_1 + 2\nu_3 + \nu_4)+\zeta_4(\nu_2 + \nu_3+ 2\nu_4)\\
 &\hspace*{0.5cm} = E_\zeta [V[\nu](y)].
\end{align*}

\subsubsection{Static versus dynamic equilibria}\label{ot-cot}

We define the type distribution $\eta$ via a parameter $0 \leq p \leq 1$. At the first step we have $\eta(x_1 = \text{E}) =  \eta(x_1 = \text{N}) = 1/2$, while the transition probabilities for the second step are
\begin{align*}
\eta(x_2 = \text{E} | x_1 = \text{E}) & = \eta(x_2 = \text{N} | x_1 = \text{N})= p \\  \eta(x_2 = \text{N} | x_1 = \text{E}) & = \eta(x_2 = \text{E} | x_1 = \text{N}) = 1-p.
\end{align*}

In the two tables below we report agents' configurations at equilibria (optimal $\pi$'s) in the two cases:
\begin{itemize}
\item[A.] agents know the full type path from the beginning (i.e.\ \emph{static} equilibria, as studied in \cite{BC16});
\item[B.] agents know the type distribution from the beginning, but the actual paths are only revealed progressively in time (i.e.\ \emph{dynamic} equilibria studied in the present paper).
\end{itemize}
We consider several values of $p$, to illustrate the effect of having full versus non-anticipative information. Mathematically, this means that the tables in B are calculated by solving problem \eqref{eq:vp}, namely
\begin{eqnarray*}
\inf_{\nu\in\Pcal(\YY^T)}\big\{\text{COT$(\eta,\nu)$}+\Ecal[\nu]\big\},
\end{eqnarray*}
 whereas those in A are computed by replacing COT with OT, namely 
 \begin{eqnarray*}
\inf_{\nu\in\Pcal(\YY^T)}\big\{\text{OT$(\eta,\nu)$}+\Ecal[\nu]\big\}.
\end{eqnarray*}

\noindent A. \emph{Static Cournot-Nash equilibria for different values of $p$.}
\begin{align*}
\resizebox{1\columnwidth}{!}{
\begin{tabular}{p{0.8cm}|p{0.8cm}p{0.8cm}p{0.8cm}p{0.8cm}}
p=0.1 & Q,Q & Q,S & S,Q & S,S \\
\hline
E,E &  0.048 &  0.001 &  0.000 &  0.000 \\
E,N &  0.073 &  0.330 &  0.008 &  0.038 \\
N,E &  0.181 &  0.006 &  0.256 &  0.008 \\
N,N &  0.000 &  0.002 &  0.009 &  0.039
\end{tabular}\quad\begin{tabular}{p{0.8cm}|p{0.8cm}p{0.8cm}p{0.8cm}p{0.8cm}}
p=0.5 & Q,Q & Q,S & S,Q & S,S \\
\hline
E,E &  0.241 &  0.007 &  0.002 &  0.000 \\
E,N &  0.042 &  0.189 &  0.003 &  0.015 \\
N,E &  0.122 &  0.004 &  0.121 &  0.004 \\
N,N &  0.003 &  0.016 &  0.042 &  0.189
\end{tabular}\quad\begin{tabular}{p{0.8cm}|p{0.8cm}p{0.8cm}p{0.8cm}p{0.8cm}}
p=0.9 & Q,Q & Q,S & S,Q & S,S \\
\hline
E,E &  0.435 &  0.013 &  0.002 &  0.000 \\
E,N &  0.009 &  0.039 &  0.000 &  0.002 \\
N,E &  0.032 &  0.001 &  0.017 &  0.001 \\
N,N &  0.011 &  0.051 &  0.070 &  0.318
\end{tabular}}
\end{align*}
\vspace*{0.05cm}

\noindent B. \emph{Dynamic Cournot-Nash equilibria for different values of $p$.}
\begin{align*}
\resizebox{1\columnwidth}{!}{
\begin{tabular}{p{0.8cm}|p{0.8cm}p{0.8cm}p{0.8cm}p{0.8cm}}
p=0.1 & Q,Q & Q,S & S,Q & S,S \\
\hline
E,E &  0.045 &  0.001 &  0.004 &  0.000 \\
E,N &  0.075 &  0.339 &  0.007 &  0.030 \\
N,E &  0.158 &  0.005 &  0.279 &  0.008 \\
N,N &  0.003 &  0.015 &  0.006 &  0.026
\end{tabular}\quad\begin{tabular}{p{0.8cm}|p{0.8cm}p{0.8cm}p{0.8cm}p{0.8cm}}
p=0.5 & Q,Q & Q,S & S,Q & S,S \\
\hline
E,E &  0.238 &  0.007 &  0.005 &  0.000 \\
E,N &  0.044 &  0.201 &  0.001 & 0.004 \\
N,E &  0.061 &  0.002 & 0.181 &  0.006 \\
N,N &  0.011 &  0.052 &  0.034 &  0.153
\end{tabular}\quad\begin{tabular}{p{0.8cm}|p{0.8cm}p{0.8cm}p{0.8cm}p{0.8cm}}
p=0.9 & Q,Q & Q,S & S,Q & S,S \\
\hline
E,E &  0.435 &  0.013 &  0.002 &  0.000 \\
E,N &  0.009 &  0.041 &  0.000 &  0.000 \\
N,E &  0.009 &  0.000 &  0.040 &  0.001 \\
N,N &  0.015 &  0.066 &  0.066 &  0.302
\end{tabular}}
\end{align*}
\vspace*{0.05cm}

As expected, knowing future types affects agents' choice in previous times as well. To wit, an agent that at the first stage already knows that she will get a parcel E at the next stage (rows 1 and 3, table A) will have more incentive to take the Q road in the first stage, as compared to the case where she only knows the current type (table B).
To see this phenomenon more clearly, in Figure~\ref{fig:bc}-l.h.s. we illustrate the different behaviour of agents of type (E,E) and (N,E) in the static and dynamic case, as function of the parameter $p$. Obviously in the two extreme situations $p\in\{0,1\}$, the knowledge of the distribution $\eta$ already reveals the full type path from the outset of the game, and thus static and dynamic equilibria coincide. The farther we are from these trivial situations, the more difference of information there is between static and dynamic case, hence the bigger the difference in the behaviour of agents at equilibrium. An analogous effect is seen in Figure~\ref{fig:bc}-r.h.s., where we report the (relative) difference between total cost in the static and dynamic equilibria.

\begin{figure}[H]
\centering
\includegraphics[scale=0.3]{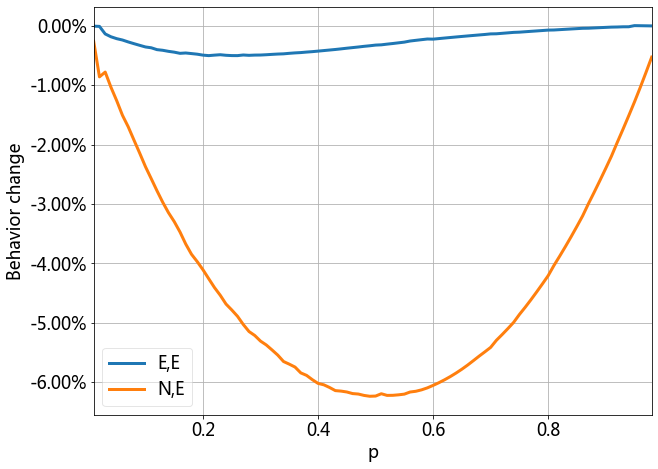}\hspace*{1cm}
\includegraphics[scale=0.3]{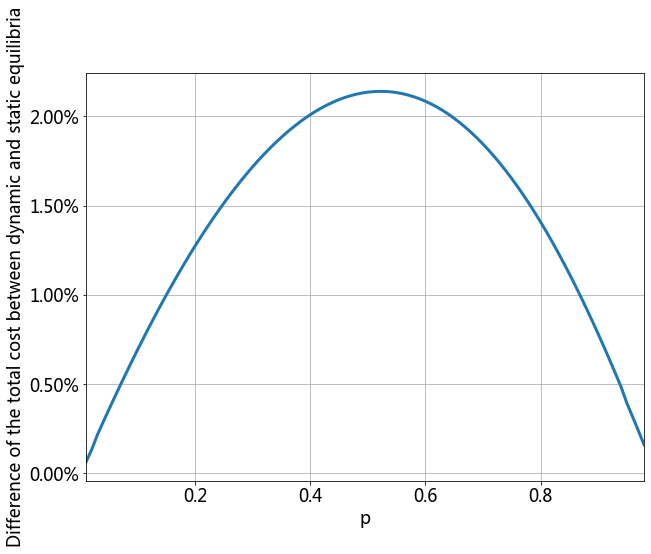}
\caption{L.h.s.: Change in the probability of agents (E,E) taking Q in the first stage, and of agents (N,E) taking Q in the first stage, when comparing static and dynamic equilibria. R.h.s.: Difference of the total cost between static and dynamic equilibria.}
\label{fig:bc}
\end{figure}

Figure~\ref{fig:bct} reports the probability of taking the Q road at stage 1 in the dynamic Cournot-Nash equilibria. Note that $p=0$ corresponds to the case where all agents switch parcel type from stage 1 to stage 2, while $p=1$ corresponds to the case of all agents getting in stage 2 the same parcel type they get in stage 1. Clearly, in the first case all agents have some incentive to take the Q road at stage 1, either because they have a parcel E at that stage, or because they will have it at the next stage, since in either case taking the S road would involve a penalty. On the other hand, in the second case, only half of the agents, the (E,E) type, have such an incentive. This observation gives the right intuition about the behaviour of agents at equilibrium, as we see from the monotonicity illustrated in Figure~\ref{fig:bct}.

\begin{figure}[H]
\centering
\includegraphics[scale=0.3]{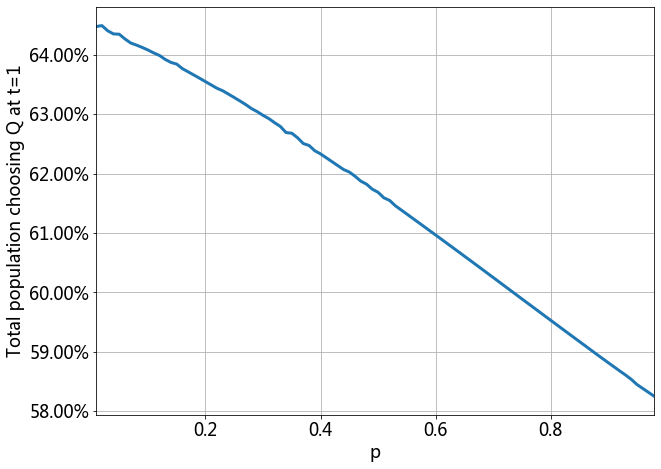}
\caption{Probability of agents taking the Q road in the first stage in dynamic Cournot-Nash equilibria.}
\label{fig:bct}
\end{figure}

\subsubsection{Price of Anarchy}\label{price-of-anarchy}
  
Let the type distribution $\eta$ be as in Section~\ref{ot-cot}. {Note that, since the functional $\Ecal$ in \eqref{eq.e.ex} is strictly convex, by Corollary~\ref{cor unique} the distribution $\nu^*$ of the optimal actions is unique, and so is the total cost associated to competitive equilibria. The latter equals  $TC[\nu^*]=\text{COT}(\eta,\nu^*)+\int V[\nu^*]d\nu^*$, and is the numerator of the PoA in \eqref{eq.poa}.
}
In Figure~\ref{fig:poa} we plot the Price of Anarchy against $p$. We notice that the smaller $p$ is, that is, the higher the probability of switching parcel type between stage 1 and 2, the higher the PoA, thus, the higher the difference between competitive and cooperative equilibria. The rationale here is that congestion at the first stage increases when $p$ becomes smaller, since in that case there is a stronger incentive to take the Quick road.

\begin{figure}[H]
\centering
\includegraphics[scale=0.3]{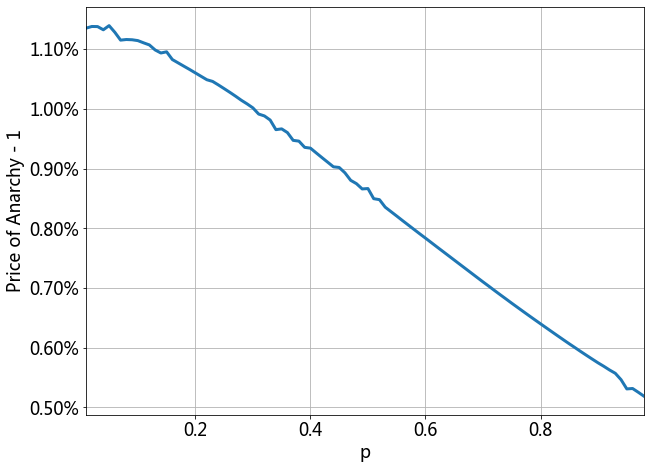}
\caption{Price of Anarchy as a function of $p$.}
\label{fig:poa}
\end{figure}

\subsection{On closed-form solutions: The Knothe-Rosenblatt case}

We will now consider a class of cost functions in discrete time for which we can characterize the optimal transport problem inside \eqref{eq:vp}, and thus mixed-strategy equilibria too, thanks to Theorem~\ref{thm:eqv}.  We denote by $F_{\mu_1}$ the cumulative distribution of $p^1_*\mu$ (with $p^1$ projection onto the first coordinate), and by $F_{\mu^{z_1,\dots,z_{t-1}}}$ the cumulative distribution  of the $t$-marginal of $\mu$ given the first $t-1$ ones, whenever $\mu$ is a measure in multiple dimensions. The increasing $T$-dimensional \emph{Knothe-Rosenblatt} rearrangement\footnote{ {The reader might know it by the name \emph{quantile transform} or \emph{Knothe-Rosenblatt coupling}.}} of $\eta$ and $\nu$ is defined as the law of the random vector $(X_1^*,\dots,X_T^*, Y_1^*,\dots,Y_T^*)$ where
\begin{align}\label{quantile transforms}\textstyle
& X_1^* = F_{\eta_1}^{-1} (U_1),\hspace{46pt}  Y_1^* = F_{\nu_1}^{-1} (U_1), \hspace{31pt}\mbox{ and inductively }\\ \nonumber
& X_t^* = F_{\eta^{X_1^*,\dots,X_{t-1}^*}}^{-1} (U_t),\quad Y_t^* = F_{\nu^{Y_1^*,\dots,Y_{t-1}^*}}^{-1} (U_t),\,\,\,\, \text{for } t=2,\dots,T,
\end{align}
for $U_1,\dots,U_T$ independent and uniformly distributed random variables on $[0,1]$. Additionally, if $\eta$-a.s.\ all the conditional distributions of $\eta$ are atomless (e.g.\ if $\eta$ has a density), then this rearrangement is induced by the (Monge) map
$$\textstyle (x_1,\dots,x_T)\mapsto A(x_1,\dots,x_T):=(A^1(x_1),A^2(x_2;x_1),\dots,A^T(x_T; x_1, \dots. x_{T-1})),$$
where $A^1(x_1):=  F_{\nu_1}^{-1}\circ F_{\eta_1}(x_1)$\ and
\begin{equation*}
\textstyle
A^t(x_t; x_1,\dots,x_{t-1}) :=\,\, F_{\nu^{A^1(x_1),\dots,A^{t-1}(x_{t-1};x_1,\dots,x_{t-2})}}^{-1}\circ F_{\eta^{x_1,\dots,x_{t-1}}}(x_t),\quad t\geq 2. 
\end{equation*}

In the proposition below we use the notation $\Delta x_t=x_t-x_{t-1}$ and $\Delta y_t=y_t-y_{t-1}$.
\begin{prop}
Assume that one of the following conditions is satisfied:
\begin{itemize}
\item[(a)] $\eta$ has independent marginals, and $f(x,y)=\sum_{t=1}^Tf_t(x_t,y_1,...,y_t)$ is such that, for all $y_1,...,y_{t-1}$, $k_t(u,z):=f_t(u,y_1,...,y_{t-1},z)$ satisfies the Spence-Mirrlees condition $\partial_{uz}k_t(u,z)<0$;
\item[(b)] $\eta$ has independent increments, and $f(x,y)=f_1(x_1,y_1)+\sum_{t=2}^Tf_t(\Delta x_t-\Delta y_t)$, with $f_t$ convex.
\end{itemize}
Then Cournot-Nash equilibria (if they exist) are determined by the second marginal, and precisely given by the Knothe-Rosenblatt rearrangement.
Moreover, if $\eta$ has a density, all Cournot-Nash equilibria are in fact pure (and given by the Knothe-Rosenblatt map).
\end{prop}
\begin{proof}
Under condition (a) or (b), Theorem~2.7 and Corollary~2.8 in [BBLZ] imply that, for any $\nu\in\Pcal(\YY^T)$, the causal transport problem COT$(\eta,\nu)$ admits a unique solution which is given by the Knothe-Rosenblatt rearrangement, and that if $\eta$ has a density, then the unique solution to COT$(\eta,\nu)$ is given by the Knothe-Rosenblatt map. 
Theorem~\ref{thm:eqv} above concludes.
\end{proof}

\section{Appendix}
\subsection{Auxiliary results and profs}

\begin{lem}\label{lem compactness}
Let $B\subseteq \Pcal(\YY^T)$ be a weakly compact set of measures, and $\eta\in \Pcal(\XX^T)$ be given. Then the set $\Pi_c(\eta,B):=\cup_{\nu\in B}\Pi_c(\eta,\nu)$ is weakly compact.
\end{lem}

\begin{proof}
Call $\tau$ and $\sigma$ the Polish topologies of $\XX^T$ and $\YY^T$, respectively. Consider 
\begin{align*}
\XX \ni x_1 & \mapsto \eta^{x_1}(dx_2,\dots,dx_T)\in \Pcal(\XX^{T-1})\\
\XX^2 \ni (x_1,x_2) & \mapsto \eta^{x_1,x_2}(dx_3,\dots,dx_T)\in \Pcal(\XX^{T-2})\\
& \vdots \\
\XX^{T-1} \ni (x_1,\dots,x_{T-1}) & \mapsto \eta^{x_1,\dots,x_{T-1}}(dx_T)\in \Pcal(\XX),
\end{align*}
for (some) regular conditional distributions of $\eta$. We can view the collection of these $T-1$ measurable mappings
as a measurable function from $\XX^T$ into a Polish space. By \cite[Theorem 13.11]{Ke95}, there is a stronger Polish topology on $\XX^T$, which we call $\hat\tau$, whose Borel sets are the same as for $\tau$, and such that the above mapping is continuous when the domain space $\XX^T$ is given the $\hat\tau$ topology.  
Let us denote by $\Sigma_1$ the topology on $\Pcal(\XX^T\times\YY^T)$ generated by convergence w.r.t.\ $\tau\times\sigma$-continuous bounded functions, and $\Sigma_2$ the topology generated by convergence w.r.t.\ $\hat\tau\times\sigma$-continuous bounded functions. 
By \cite[Proposition 2.4]{BBLZ} we know that causality can be tested by integration against functions of the form
$$h(y_1,\dots,y_t)\left [ g(x_1,\dots,x_T)- \int g(x_1,\dots,x_t,\bar x_{t+1},\dots,\bar x_T)\eta^{x_1,\dots,x_t}(d\bar x_{t+1},\dots,d\bar x_T) \right ],$$
for each $t$, $h$ bounded $\sigma$-continuous and $g$ bounded $\tau$-continuous. Notice that the function in brackets is then by definition also $\hat\tau$-continuous, so the overall expression is $\hat\tau\times\sigma$-continuous. It follows that $\Pi_c(\eta,B)$ is $\Sigma_2$-closed. On the other hand, $\Pi_c(\eta,B)$ is also $\Sigma_2$-tight, since as a Borel measure $\eta$ is still tight w.r.t.\ the stronger topology induced by $\hat\tau$-continuous bounded functions. Thus $\Pi_c(\eta,B)$ is $\Sigma_2$-compact and in particular also $\Sigma_1$-compact.  
\end{proof}

\begin{proof}[Proof of Theorem \ref{thm_ex_dua}]
The existence follows from Lemma \ref{lem compactness}, since $\Pi_c(\eta,\nu)$ is weakly compact, and the functional $\pi\mapsto \int fd\pi$ is lower semicontinuous. For the duality, me may first stregthen the topology on $\XX^T$, just as we did in the proof of Lemma \ref{lem compactness}, so guaranteeing that the conditional distributions $\{\eta^{x_1,\dots,x_t}\}_t$ are continuous. Doing so shows that each $F\in\FF$ is continuous after strengthening the topology. Similarly, the proof of \cite[Proposition 2.4]{BBLZ} reveals that the martingales $\{M_t\}_t$ are determined by the conditional distributions $\{\eta^{x_1,\dots,x_t}\}_t$, and so each $S\in\SS$ can be assumed continuous after strengthening the topology. Crucially, since $\eta$ remains a Borel measure after this strengthening of topology, the set $\Pi(\eta,\nu)$ is still compact after we accordingly strengthen the weak topology on $\Pcal(\XX^T\times\YY^T)$. This proves that
\begin{align*}\text{COT}(\eta,\nu)&= \inf_{\pi\in \Pi(\eta,\nu)}\sup_{S\in\SS} \int(f+S)d\pi  = \sup_{S\in\SS}\inf_{\pi\in \Pi(\eta,\nu)} \int(f+S)d\pi\\& = \inf_{\pi\in \Pi(\eta,\nu)} \sup_{F\in\FF} \int(f+F)d\pi = \sup_{F\in\FF}\inf_{\pi\in \Pi(\eta,\nu)} \int(f+F)d\pi, 
\end{align*}
by Proposition \ref{equivalence-proposition} and a familiar application of Sion's minimax theorem. The remaining identities are obtained by applying Kantorovich duality (cf.\ \cite{Villani}) to $\inf_{\pi\in \Pi(\eta,\nu)} \int(f+S)d\pi $ and $\inf_{\pi\in \Pi(\eta,\nu)} \int(f+F)d\pi$.
\end{proof}

\begin{lem}
\label{non-linear-duality}
If $f$ is lower bounded and lower semicontinuous, then
$$ \inf_{\pi\in\Pi_c(\eta,\nu)} \left\{{\textstyle \int f\ d\pi} - \epsilon\, \text{Ent}(\pi) \right\} = \sup_{S\in\SS}\, \text{OT}^{\epsilon,S}  .$$
\end{lem}

\begin{proof}
Very similar to the proof of Theorem \ref{thm_ex_dua}, where Sion's minimax theorem is invoked. We only have to check that $-\text{Ent}(\cdot)$ is convex and lower semicontinuous. Letting $g(x):=x\,\log(x)$, which is a convex continuous function on $[0,\infty)$, this immediately implies that $\pi\mapsto -\text{Ent}(\pi)= \sum g(\pi_{i,k})$ is convex and continuous.
\end{proof}

\subsection{Basis for $\SS$}\label{app.basis}
We first get a better understanding of $x$-adapted and $y$-adapted processes, as introduced at the beginning of Section~\ref{sect.num}, in the finite setting of Assumption~\ref{ass.fin}, where $|\XX|=n$ and $|\YY|=m$.
Consider the case where we have three-steps, $T=3$, and two possible actions, say $\YY=\{a,b\}$. Then the path space of actions $\YY^T$ can be represented with the tree in Figure~\ref{fig:tree}.

\begin{figure}[H]
\centering
\includegraphics[scale=0.3]{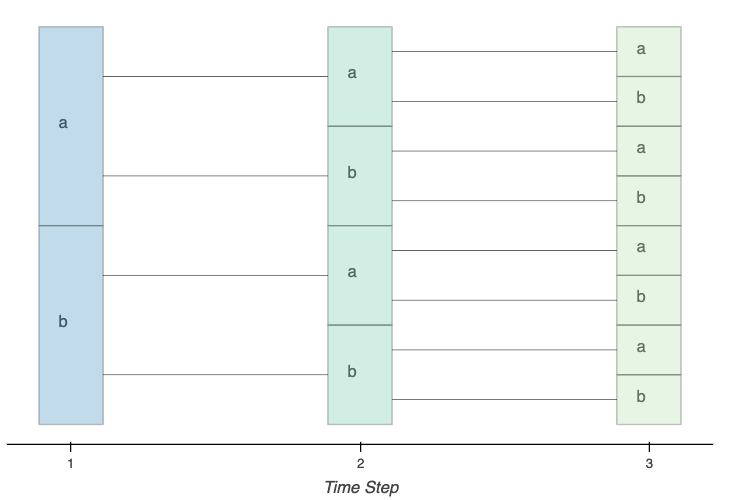}
\caption{Tree representation of the path space of actions}
\label{fig:tree}
\end{figure}

We call `node' every box of the tree, and `leaf' any node at the terminal time $T$. Now, any $y$-adapted process can be represented with a tree with same structure, while having any possible values in the nodes. Clearly, a basis for such processes can be given by assigning the value 1 to any of the nodes of the tree and 0 to all other nodes. This means that $y$-adapted processes can be identified with elements in $\RR^{\#\text{nodes}}$. In general, for any number of steps $T\in\{2,3,\ldots\}$ and any number of actions $m$, $y$-adapted processes are identified with vectors in $\RR^{\#\text{nodes}}=\RR^{N_{\YY}}$, with $N_{\YY}=\sum_{t=1}^Tm^t$. 

Analogously, we can represent the path space of types $\XX^T$ with a tree, which defines the structure of all $x$-adapted processes. Now, to build the set $\SS$, we are only interested in $x$-adapted processes that are $\eta$-martingales. For these processes, knowing the values at terminal time is enough, since values at previous times are then determined by backward recursion thanks to the martingale property. A basis for such processes can therefore be given by assigning 1 to any of the leaves of the tree representing $\XX^T$, and 0 to all other leaves. $x$-adapted $\eta$-martingales can therefore be identified with vectors in $\RR^{\#\text{leaves}}=\RR^{n^T}$.

Finally, $\SS$ is generated by the finite basis of elements of the form $\sum_{t<T}h^i_t(M^j_{t+1}-M^j_t)$, where $\{h^i\}_{i=1}^{N_{\YY}}$ is a basis for the $y$-adapted processes and $\{M^j\}_{j=1}^{n^T}$ is a basis for the $x$-adapted $\eta$-martingales. We denote such a basis for $\SS$ by $\{e_k\}_{k=1}^K$, where $K=n^T\cdot N_{\YY}$.

\bibliographystyle{plain}
\bibliography{biblio_CN2}

\end{document}